\documentclass[a4paper,reqno]{amsart}

\usepackage[mathscr]{eucal}
\usepackage{mathrsfs}
\usepackage{mathbbol}
\usepackage{oldgerm,units}
\usepackage{wrapfig}
\usepackage{epsfig}
\usepackage{epsfig,latexsym,amsfonts,amssymb,amsmath,amscd,graphics,epic}
\usepackage{amsfonts,amssymb,amsmath,amscd,amsthm}
\usepackage[mathscr]{eucal}
\usepackage{mathrsfs}
\usepackage{mathbbol}
\usepackage{oldgerm,units}
\usepackage{enumerate}
\usepackage{vaucanson-g}
\usepackage{tikz}
\usetikzlibrary{arrows}
\usepackage{xcolor}

\def\SG{\mathcal S}
\newcommand{\ds}[1]{\ {#1} \ }

\newtheorem{theorem}{Theorem}[section]
\newtheorem{proposition}[theorem]{Proposition}
\newtheorem{definition}[theorem]{Definition}

\newtheorem{lemma}[theorem]{Lemma}

\newtheorem{corollary}[theorem]{Corollary}

\newtheorem{example}[theorem]{Example}

\newtheorem{problem}[theorem]{Problem}

\newcommand {\junk}[1]{}

\def\iff{\Longleftrightarrow}

\def\Im{\operatorname{Im}}

\def\({\left(}
\def\){\right)}

\def\implies{\Rightarrow}
\def\al{\alpha}

\def\nxn{n\times n}

\def\w{k}
\def\limw{\lim_{\w}}

\def\minf{-\infty}

\def\cF{\mathcal F}
\def\cG{\mathcal G}
\def\SG{\mathcal S}
\def\gen{\mathcal I}
\def\T{\mathbb T}
\def\P{\mathbb P}
\def\N{\mathbb N}
\def\R{\mathbb R}
\def\Gc{{\mathcal G}_c}
\def\G{\mathcal G}
\def\Tnn{\T^{\nxn}}

\def\eref#1{(\ref{#1})}
\def\proj#1{\pi \langle #1 \rangle}
\def\closproj#1{\overline{\pi\langle #1\rangle}}

\def\per{\operatorname{per}}

\def\cyc{\operatorname{cyc}}
\def\diag{\operatorname{diag}}

\def\det{\operatorname{det}}

\def\rk{\operatorname{rk}}
\def\urk{\operatorname{urk}}
\def\tr{\operatorname{tr}}

\def\sh{\operatorname{S/B}}
\def\GMr{\operatorname{G-M; rw}}
\def\GMc{\operatorname{G-M; cl}}
\def\sym{\operatorname{sym}}
\def\kpr{\operatorname{Kp}}
\def\row{\operatorname{rw}}
\def\clm{\operatorname{cl}}
\def\trrk{\rk_{\tr}}

\def\shrk{\rk_{\sh}}
\def\gmrrk{\rk_{\GMr}}
\def\gmcrk{\rk_{\GMc}}
\def\symrk{\rk_{\sym}}
\def\kprk{\rk_{\kpr}}
\def\rwrk{\rk_{\row}}
\def\clrk{\rk_{\clm}}

\def\fcell{\operatorname{\mathcal F}}

\def\minf{-\infty}

\def\bfx{{\bf x}}
\def\bfy{{\bf y}}
\def\bfu{{\bf u}}
\def\bfv{{\bf v}}

\def\tlA{\widehat A}

\def\ie{\textit{i.e.}, }

\begin{document}
\title[The ultimate rank of tropical matrices] {The
  ultimate rank of tropical matrices}

\author{Pierre Guillon}
 \address{CNRS et  Institut de Math\'{e}matiques de Luminy, IML  Case 907, Campus de Luminy, 13009 Marseille, France.}
 \email{pguillon@math.cnrs.fr}

\author{Zur Izhakian}
\address{School of Mathematical Sciences, Tel Aviv
     University, Ramat Aviv,  Tel Aviv 69978, Israel.}
\email{zzur@math.biu.ac.il}

\author{\\ Jean Mairesse}
 \address{CNRS, UMR 7089, LIAFA, Univ Paris Diderot, Sorbonne Paris Cit\'e, F-75205 Paris, France.}
\email{mairesse@liafa.univ-paris-diderot.fr}

\author{Glenn Merlet}
 \address{Universit\'ee d'Aix-Marseille,
IML, Case 907,
Campus de Luminy, 13009 Marseille, France}
 \email{glenn.merlet@univ-amu.fr}


%******************************* date *************************************
\date{\today}
\thanks{\noindent \underline{\hskip 3cm } \\ File name: \jobname}

%******************************* keywords *********************************


%******************************* abstract *********************************

\begin{abstract}
A tropical matrix is a matrix defined over the  max-plus semiring.
For such matrices, there exist several non-coinciding notions of rank:
the row rank, the column rank, the Schein/Barvinok rank, the Kapranov
rank, or the tropical rank, among others. In the present paper, we show that there exists a
natural notion of ultimate rank for the powers of a tropical matrix,
which does not depend on the underlying notion of rank.  Furthermore, we provide a simple
formula for the ultimate rank of a matrix which can therefore be computed
in polynomial time.
Then we turn our attention
to finitely generated semigroups  of matrices, for which our notion of
ultimate rank is generalized naturally. We provide both combinatorial and
geometric characterizations of semigroups having maximal
ultimate rank.
As a byproduct, we obtain a polynomial algorithm to decide if the
ultimate rank of a finitely generated semigroup is maximal.
\end{abstract}

\maketitle

%%%%%%%%%%%%%%%%%%%%%%%%%%%%%%%%% section %%%%%%%%%%%%%%%%%%%%%%%%%%%%%

{\small \tableofcontents}

%%%%%%%%%%%%%%%%%%%%%%%%%%%%%%%%% section %%%%%%%%%%%%%%%%%%%%%%%%%%%%%

\section{{\bf Introduction}}
\numberwithin{equation}{section}

 Tropical matrices are matrices defined over the so-called
 ``max-plus'' or ``tropical'' semiring, that is,
 $(\R\cup\{-\infty\},\max, +)$. Tropical algebra is a
 rapidly growing area and surveys with different viewpoints are available,
 e.g. \cite{BCOQ,Butkovic,IMSh}. 

The tropical semiring is a ``weak'' algebraic structure in which inverses do not
 exist for the $\max$ operation. Tropical matrices essentially correspond
 to weighted digraphs; as
 such, the interplay between algebra and graph theory,
 % takes a
 % fundamental role in the study of tropical matrices
 % \cite{Butkovic}. Digraphs are therefore a central computational tool
 % in tropical matrix algebra \cite{SeSc},
 and combinatorial ideas, often have helped to bypass the lack of basic
 algebraic methods in the tropical framework.
 % ~\cite{RSTh}.

Yet, negation is still absent and  the elementary process of Gauss
elimination is not feasible over the tropical  setting; as a
consequence, familiar algebraic concepts, such as dependence and
spanning, basically do not agree here. This drawback has led to
various notions of matrix ranks (the row rank, the column rank, the
Schein/Barvinok rank, the Kapranov
rank, or the tropical rank, among others,
cf. Definition~\ref{de:ranks}) which, as one would expect,
essentially do not coincide.  Over the
years, much effort has been invested in the study of the different
types of ranks and the relations among them, see for instance \cite{AGGu,DSSt}.
However the use of these ranks for applications, especially for computation, is rather cumbrous.

In the present paper we introduce a canonical notion, termed ultimate
rank, that in a sense unifies the known ranks of tropical matrices.
The ultimate rank of a matrix is defined as the minimal rank over the
closure of the semigroup generated by the matrix, and the minimum is
proven not to depend on the underlying notion of rank.
%This rank has a natural combinatorial meaning in terms of associated digraphs of matrices.
The ultimate rank of the matrix $A$ depends only on the so-called
critical graph (see Definition~\ref{de:c(a)}), with the following simple
and explicit  formula (Theorem~\ref{th:OneGenerator}): % (cf. \eqref{eq:formula}):
\begin{equation*}
\urk(A) = \sum_{C\in  {\mathfrak C}} \cyc(C) \:,
\end{equation*}
where
${\mathfrak C}$ is the set of strongly connected components of the
critical graph of $A$, and where
$\cyc(C)$ is the cyclicity of $C$.
Therefore, unlike some of the known ranks,  the ultimate rank  can be
computed in polynomial  time-complexity, $O(n^3)$ for an $n \times n$  matrix (Corollary~ \ref{co:pol1}).
The proof of the formula relies on the \emph{ultimate expansion} of
tropical matrices \cite{SeSc}.
% For an easy exposition, we first prove the formula for idempotent
% matrices (\ie matrices $E$ such that $E^2=E$). Then, in
% Lemma~\ref{le:idem},  we appeal to general matrices for which we are
% assisted by the visualization principle (Theorem \ref{th:visual})
% and the \emph{ultimate expansion} of tropical matrices
% \cite{SeSc}. As a consequence, we get  that computing the ultimate
% rank of an $n \times n$  matrix is of time-complexity $O(n^3)$
% (Corollary  \ref{co:pol1}).

The next step is to
generalize the notion of ultimate rank to a finitely generated semigroup of
tropical matrices, defined as the minimal rank over all the matrices
in the closure of the semigroup. Determining the ultimate rank of a given
matrix semigroup is then
the obvious question. Unfortunately, the algorithmic computability of
the ultimate rank is an open question. We settle the case of maximal
ultimate rank. 
Indeed, we provide two different characterizations,
one combinatorial and  one geometric, for finitely generated matrix semigroups
of maximal ultimate rank. 
The combinatorial characterization (Theorem~\ref{Thm:Comb}) is
provided  in terms of the structure of the associated graphs of the
generators.
The geometric characterization (Theorem~\ref{th:geom}) is the existence of a common (tropical) eigenvector for all the
generators which lies in their fundamental cells.
Having these characterizations, we show that the problem of
determining whether a semigroup of $n \times n$
matrices has maximal ultimate rank  is decidable in time complexity $O(|\gen|
n^3)$, where $\gen$ is the set of generators.

% Our exposition throughout the paper is supported by several examples,
% demonstrating the various cases of the
% geometric and combinatorial characterizations. Open questions
% are discussed in Section \ref{sec:summery}.

A by-product of the present paper is to provide new insight on matrices and
semigroups of matrices with
maximal tropical rank. 
Matrices with maximal tropical rank are also
known as {\em non-singular} or {\em strongly regular} matrices. They
have been extensively studied in the literature, see for instance
\cite{AGGu,BuBu,butk94,butk00,BuHe,DSSt,I,IR1,IR2,KiRo05,merl10,RSTh}. 
They play a  key role in solving linear systems of tropical equations, in
studying the optimal assignment problem, or in the theory
of discrete event systems.

\newpage

\section{{\bf Preliminaries}}

The {\em max-plus semiring} or {\em tropical semiring} $(\T,\vee,+)$ is the
set $\T=\R \cup \{ \minf \}$ equipped with the binary operations $(x,y)
\mapsto \max(x,y)=x\vee y$ and  $(x,y)
\mapsto x+y$. In this structure, the ``additive'' operation is the
maximum while the ``multiplicative'' operation is the usual sum; their
respective identity elements are $-\infty$ and~$0$.

Classical algebraic objects, such as vectors and matrices,  have max-plus analogues.
The set  $\T^n$ of
$n$-dimensional {\em vectors} is a semimodule over the semiring
$\T$.
The set  $\T^{n\times m}$ of $n \times m$ {\em matrices} is a
semimodule over $\T$.
A vector  $\bfv$ having entries $v_i\in \T$, resp. a matrix $A$ having entries
$A_{ij} \in \T$, is written as $\bfv = (v_i)$, resp. $A =
(A_{ij})$.  The matrix with only $-\infty$ entries, written $(-\infty)$ is the \emph{null} matrix (of a
given size).

The product of matrices is defined according to the semiring structure
and is denoted by using the special symbol $\odot$\;, that is, for matrices $A,B$ of
compatible sizes, $A\odot B$ is defined by
\[
(A\odot B )_{ij} = \bigvee_{k} \bigl(A_{ik}+ B_{kj} \bigr)\:.
\]
The pair
$(\T^{\nxn},\odot)$ forms a monoid.

It is convenient to use the following notations. As usual,
we write $AB$ for the product $A
\odot B$. For matrices $A, B$ of the same size, we define $A\vee B$ by $(A\vee B)_{ij}=A_{ij}\vee
  B_{ij}$ and write $A \geq B$ if $A_{ij} \geq B_{ij}$ for every $i,j$ (that is  $A \vee B = A$).   The product $\lambda\odot A$ of a matrix $A$ and a scalar $\lambda\in\T$ is  defined by
$(\lambda\odot A)_{ij}= \lambda + A_{ij}$.

%\medskip

\begin{definition}\label{de:tropequiv}
Two matrices $A$ and $B$ of the same size are said to be {\em (tropically)
  equivalent} if there exists
 $\lambda$ in $\R$ such that $B=\lambda\odot A$.
Let $\P\T^{n\times m}$ be the {\em (tropical) projective} matrix set obtained
 by identifying
  equivalent matrices in $\T^{n\times m}$.
The quotient map
 is denoted by $\pi: \T^{n\times m} \rightarrow \P\T^{n\times
   m}$.
\end{definition}

%\medskip
Graph theory provides an important tool in the study of tropical
matrices, established via the following correspondence.

\begin{definition}\label{de:g(a)}
The {\em graph} of a matrix  $A \in \T^{\nxn}$, denoted by ${\cG}(A)$,
is the weighted directed graph with nodes  $\{1,\dots ,
n\}$ and arcs  $(i,j)$ whenever  $A_{ij}\neq -\infty$. The {\em
  weight} of arc $(i, j)$ is $A_{ij}$.
% We say that $A$ is {\em
%   irreducible} if ${\cG}(A)$ is strongly connected, and {\em
%   reducible} otherwise.
\end{definition}

Powers of a max-plus matrix $A$ can be interpreted
combinatorially in terms of directed paths in
${\cG}(A)$. Let a {\em walk} in ${\cG}(A)$ be a finite sequence of nodes
$p=(i_0,i_1,\cdots, i_{\ell})$ such that $(i_k,i_{k+1})$ is an arc of~${\cG}(A)$ for any~$k$.
Its {\em  length} is $\ell$ and its {\em weight}
is~$0$ if $\ell=0$ and $A_{i_0i_1} + \cdots + A_{i_{\ell-1}i_{\ell}}$ otherwise.
For $\w\in \N$, the
entry $(A^{\w})_{ij}$ of $A^\w$ is equal to the maximal weight of a walk of
length $\w$ from $i$ to $j$  in
${\cG}(A)$.

In this paper, we consider only directed graphs, and use the
classical terminology of graph theory.
%A {\em walk} is a directed path.
A walk from $i$ to itself is called a {\em circuit}, and a circuit of length $1$ is called a {\em loop}.
Moreover, a circuit is called {\em simple} if it contains each
node at most once, except for the first and last one that are the
same.
A graph is {\em strongly connected} if there is a walk from each vertex in the graph to every other vertex.
A {\em strongly
  connected component} (written {\em scc}, for short) of a graph is a subgraph which
is strongly connected and maximal with respect to inclusion.
A scc is {\em trivial} if it consists of one node with no loop.
A graph is {\em completely reducible} if it is a disjoint union of scc.
The {\em cyclicity} of a strongly connected graph is the gcd of the lengths of its circuits.
For a completely reducible graph, the {\em cyclicity} is the lcm of the cyclicities of the scc.

A matrix $A$ is called {\em irreducible} if its associated graph
${\cG}(A)$ is strongly connected; otherwise $A$ is called {\em reducible}.

\section{{\bf The different ranks of a matrix}}

There exist several relevant notions of rank for matrices over the tropical semi\-ring, see
\cite{DSSt} and \cite{AGGu} for extensive accounts. In
Definition~\ref{de:ranks} below, we recall the main ones.
To this aim, we first review some necessary concepts.

%\medskip

\begin{definition} $ $
\begin{enumerate}
 \item The {\em (tropical) permanent} of a matrix $A\in \T^{\nxn}$ is
defined by:
  \begin{equation}\label{eq:tropicalDet}
 \per(A) = \bigvee_{\sigma \in \mathfrak{S}_n}    A_{1\sigma(1)} +
 \cdots +  A_{n\sigma(n)},
\end{equation}
where $\mathfrak{S}_n$ is the set of all the permutations on $\{1,\dots,
  n\}$.
 \item A matrix $A \in \T^{\nxn}$ is called \emph{(tropically) singular}
if $A=(-\infty)\in \T^{1\times 1}$ or if there exist at least two different permutations that reach
the maximum in $\per(A)$, that is,
$$ \per(A)  =  \sum_i A_{i\sigma(i)} = \sum_i A_{i\tau(i)}
 \:,$$
for some  $\sigma \neq \tau$ in $\mathfrak{S}_n$.
Otherwise $A$ is called \emph{(tropically) non-singular}.

If $A$ is non-singular, we denote by $\tau_A$ the unique permutation that
reaches the maximum in (\ref{eq:tropicalDet}).
\end{enumerate}
\end{definition}

Non-singularity is equivalent to \emph{strong regularity} in the sense of~\cite{butk94}. In the sequel, we will use the following result.

\begin{proposition}\label{pr:merlet} Given  $A, B
  \in \T^{\nxn}$ if $AB$ is non-singular
  then $A$ and $B$ are non-singular, and:
\begin{equation*}
%  \item $(AB)_{\tau_{AB}} = (A)_{\tau_{A}}   (B)_{\tau_{B}}$;
{\tau_{AB}} = \tau_B \circ \tau_A, \qquad
\per(AB) = \per(A) + \per(B) \:.
\end{equation*}
\end{proposition}

This is proven in \cite[Proposition 3.4]{merl10}. % for
% matrices with at least one finite entry on each row, but the proof can be easily extended.
%(The right equality was already proved in \cite{DSSt} and \cite[Theorem 3.5]{IzhakianRowen2008Matrices}.)

\begin{definition}
A family of vectors $\bfx_1, \dots, \bfx_r \in
    \T^{n}$ is {\em linearly independent} (in the Gondran-Minoux sense)
    if for all disjoint $I,J \subset \{ 1, \dots, r\}$, and for all
    $\al_i \in
    \T, i\in I\cup J$, with $(\al_i)_i\neq (-\infty,\dots, -\infty)$, we have $\bigvee_{i\in I}
    (\al_i\odot
    \bfx_i ) \neq \bigvee_{ j\in J} (  \al_j\odot
    \bfx_j ) $.
\end{definition}

%\medskip
\begin{definition} $ $
\begin{enumerate}
 \item The  set $X \subset \T^n$ is {\em tropically convex} if: $\forall \bfu, \bfv \in X, \ \forall \lambda, \mu \in \T$,
\begin{equation}\label{eq:convex}
(\lambda\odot \bfu ) \vee
(\mu\odot \bfv ) = \bigl( (\lambda,\dots , \lambda) + \bfu \bigr) \vee \bigl(
(\mu, \dots, \mu) + \bfv \bigr) \ \in \  X \:.
\end{equation}
 \item The tropically convex set generated by a finite family $\gen$ of vectors
of $\T^n$, that is the set
\[
\bigg\{ \ \bfy \in \T^{n} \ \mid \
\exists (\al_{{\bf s}})_{{\bf s} \in \gen}, \al_{{\bf s}}\in \T, \quad \bfy =
\bigvee_{{\bf s}\in \gen} ( \al_{{\bf s}} \odot {\bf s} ) \ \bigg\}\:,
\]
is called the \emph{(tropical) convex hull} of $\gen$.
\end{enumerate}
\end{definition}

A tropically convex set is invariant by translations along the
direction $(1,\dots, 1)$. Therefore, it can be  identified  with its image
in the projective space $\P\T^n$.
Next result is proven in \cite{wagn91}.

\begin{proposition}
Every finitely generated tropically convex set has a projectively unique generating
set which is minimal for inclusion.
%The cardinality of the minimal generating set is called the {\em weak dimension}.
% and is  denoted by $\weakdim(\cdot )$.
\end{proposition}

\begin{definition}\label{de:topdim}
The {\em weak dimension} of a finitely generated tropically convex
subset $X$ of
$\T^n$ is the cardinality of the minimal generating set of $X.$

The {\em topological dimension} of a subset $X$ of $ \T^n$ is the largest $k$ for which  there exists an affine space $K \subset \R^n$
of dimension $k$, such that $X\cap K$ has a non-empty relative interior in $K$.
\end{definition}

Below, we apply this notion of topological dimension to finitely generated tropically convex
subsets of~$\T^n$, for which it coincides with the other classical notions of dimensions, such as 
the Hausdorff dimension or the Lebesgue covering dimension. 

Consider the  tropically convex set of Figure
\ref{fi-topweak}. The topological dimension is 2 since the set is the
union of three infinite strips. The weak
dimension is 3 since the convex hull of any two finite vectors is at
most the union of two infinite strips. 

In general, for a given $\T^n$, $n\geq 3$, the weak dimension is unbounded whereas
the topological dimension is bounded by $n$.

% To see the analogy with the classical convexity, rewrite
% \eref{eq:convex} using the max-plus notations:
% $\forall \bfu, \bfv \in E, \ \forall \lambda, \mu \in \T,
% \ (\lambda\odot \bfu ) \vee (\mu\odot \bfv ) \in
% E$.

\medskip

Let us define the main notions of rank for tropical matrices.

\begin{definition}\label{de:ranks}
Consider a matrix $A\in \T^{n\times m}$. 

\begin{description}
    \item[Tropical rank] Let $\trrk(A)$ be the maximal $r$
    for which there exists an $r \times r$ non-singular submatrix of
    $A$.

%  \item[{Supertropical rank}] A family of vectors $\bfx_1, \dots, \bfx_r \in
%     \T^{(n)}$ is linearly dependent in supertropical sense if there are
%     $\al_1,\dots,\al_r$,
%     not all of them $\minf $, such that each coordinate of $\bfy :=  \bigvee_i \{ \al_i + \bfx_i \} $
%      is attained by at least two different competent $ \al_i + \bfx_i$'s in the equation; otherwise
%      the family      is said to be independent.
% The supertropical rank  $\strrk(A)$, is defined to be the maximal
% number of independent rows of $A$. (Taking this definition with
% respect to the columns of $A$, one receives identical rank, cf.
% \cite{IzRo}.)

\item[Symmetrized rank] Define $\det^+(A)$ (resp. $\det^-(A)$) as in \eref{eq:tropicalDet} but with $\sigma$
      ranging over permutations of even (resp. odd) sign.
The symmetrized rank $\symrk$  is the maximal $r$
    such that $A$ has an $r \times r$ submatrix $B$ for which $\det^+(A)\ne
    \det^-(A)$.

    \item[Gondran-Minoux rank] The Gondran-Minoux row rank $\gmrrk(A)$ is the maximal $r$
such that $A$ has $r$ independent rows. The Gondran-Minoux column rank
$\gmcrk(A)$ is defined similarly with respect to the columns.

    \item[Kapranov rank] The Kapranov rank $\kprk(A)$ is defined for
      instance in \cite[Def. 1.2 and 3.2]{DSSt}.

% be the smallest topological dimension of any
% tropical linear space containing the columns of $A$, cf.
% \cite[Definition 1.2]{DSSt}.

    \item[Schein/Barvinok rank] Let  $\shrk(A)$ be the minimal
      $r$ such that: $\exists B \in \T^{n\times r}, C\in \T^{r\times
        m}$, $A=BC\in\T^{n\times m}$.

% smallest integer $r$ for which
% $A$ can be written as the tropical sum of $r$ matrices, each of
% which is the tropical product of a row vector and a column vector, or equivalently
%     $$\shrk(A) = \min \{ r \geq  1 \ | \ A= BC, \text{ where  }
%     B \in M_{n \times r} (\T) \text{ and } C \in M_{r \times
%     m}(\T)\}.$$

    \item[Row rank] Let $\rwrk(A)$ be the weak dimension of the convex hull of the row vectors of $A$.

    \item[Column rank]  Let $\clrk(A)$  be the weak dimension of the convex hull
     of the column vectors of
      $A$.
\end{description}
Two equivalent matrices have the same rank for any of the
above notions. Therefore, the different notions of ranks can be viewed
as being defined on $\P\T^{n\times m}$.
\end{definition}

By convention, all the ranks of the null matrix $(-\infty)$ are
set to 0. For a matrix $A\neq (-\infty)$, we check that
$\rk_{\star}(A)>0$ for any of the above notion of rank.

\medskip

None of the above notions coincide~\cite[Section 8]{AGGu}.
The following relations have been established, see \cite[Theorem
8.6]{AGGu} for \eref{eq:relation1} and \cite[Theorem 1.4]{DSSt} for 
\eref{eq:relation2}:

\begin{equation}\label{eq:relation1}
 \trrk(A) %= \strrk(A)
  \leq \symrk(A) \leq \left\{ \begin{array}{l}
                          \gmrrk(A)  \\  \gmcrk(A)  \\ \end{array}   \right \}
         \leq \shrk(A) \leq \left\{ \begin{array}{l} \rwrk(A)\\
                                                  \clrk(A)\\
                                                  \end{array}
                                                  \right.
\end{equation}
and  \begin{equation}\label{eq:relation2} \trrk(A) \leq \kprk(A)
         \leq \shrk(A) .
\end{equation}

The extremal values of the ranks are of specific interest.
\begin{description}
\item[Rank 0] Any of the ranks is $0$ iff the matrix is null.
\item[Rank 1] It can easily be checked that $\trrk(A)= 1$ iff $A$ is
  non-null and all the non-null rows are tropically
equivalent. Thus, rank~$1$ occurs (or not) simultaneously for all ranks.

\item[Maximal rank] Consider $A\in \T^{\nxn}$. According to
\eref{eq:relation1} and \eref{eq:relation2}, we have:
$\trrk(A) =n \Rightarrow \rk_{\star}(A)=n$,
for any of the above notions of rank. It corresponds to the case of a
non-singular matrix.
\end{description}

%\medskip

Below, we focus on the extremal ranks given by \eqref{eq:relation1} and \eref{eq:relation2}, that is, the column rank, the
row rank, and the tropical rank.
It turns out that these three notions of rank are related to the dimension of the ``image'' of the
matrix. It follows directly from the definition for the column and row
ranks, and from the next result for the tropical rank.

 \begin{proposition}\label{pr-sturmfels}
 Consider a matrix $A\in \T^{n\times n}$ such that: $\forall i,
 \exists j,  A_{ij}\neq -\infty$. Then $\trrk(A)$
is the topological dimension of the convex hull
     of the column (resp. row) vectors of $A$.
 \end{proposition}

Proposition \ref{pr-sturmfels} appears in \cite[Theorem 4.2]{DSSt}
where the proof is carried out for  matrices in $\R^{n\times n}$ but
can be easily adapted. See also \cite[Theorems 3.3 and
4.1]{butk00}. 

\medskip

The following examples show that
these 3 notions of rank do not coincide.

\begin{example}\label{ex:1}
Consider the matrices:
\[
A= \left[ \begin{array}{rrr} -1 & 0 & 0 \\ 0 & -1 & 0  \\ 0 & 0 & -1
\end{array}\right], \qquad
B= \left[ \begin{array}{cccc} 2 & 1 & 0 & -1 \\ -2 & 1 & 0 & 1  \\ 0 & 0
    & 0 & 0 \\ 0 & 0 & 0 & 0
\end{array}\right] \:,
\]
for which we have:
\begin{eqnarray*}
\trrk(A)=2, & \rwrk(A)=3, & \clrk(A)=3, \\
\trrk(B)=3, & \rwrk(B)=3, & \clrk(B)=4 \:.
\end{eqnarray*}
Let $\Im(A)$ be the convex hull of the
columns of $A$ (i.e. the image set of the mapping: $\T^3\rightarrow \T^3,
{\bf x} \mapsto A\odot {\bf x}$).
In Figure \ref{fi-topweak}, we have represented the set $\Im(A)$ on the
left, and the projective set $\pi(\Im(A))$ on the right (represented by orthogonal
projection on the plane orthogonal to the direction $(1,1,1)$).

\tikzset{math3dMP/.style={y={(0.4cm,1cm)},z={(-0.1cm,0.7cm)},x={(1cm,0.1cm)}}}

\begin{figure}[ht]
 \caption{The tropically convex set $\Im(A)$ (on the left) and its projective image $\pi(\Im(A))$ (on the right).}
 \label{fi-topweak}

\begin{tikzpicture} [math3dMP,scale=1]

%\clip (-1.2,-2) rectangle (1.3,2.4);

%	\draw[->,thick] (0,0,0) -- (1,1,1) node[above]{$\vec{S}$};

	\coordinate (a) at (-2,-1,-1) ;
	\coordinate (b) at (-1,-2,-1) ;
	\coordinate (c) at (-1,-1,-2) ;
	\coordinate (o) at (0,0,0) ;
	\coordinate (e) at (-4/3,-4/3,-4/3) ;
	\coordinate (d) at (8/3,8/3,8/3) ;
	\coordinate (od) at (4/3,4/3,4/3) ;
	\coordinate (t) at (1,1,1) ;
	\coordinate (t2) at (5/3,5/3,5/3) ;
%	\draw (a) -- ++(d) -- (d) -- (o) -- (a);
	\filldraw[fill=gray!50, opacity=0.8, dashed ] (c) -- ++(d) -- (od) -- (e) -- cycle;
	\draw[very thick] (c) -- ++(t) node{$\bullet$} node[right]{} -- ++(t2);
	
	\draw[->,Magenta,>=angle 45] (0,-2.5,0) -- (0,2,0) ;%node[right]{$y$}
	
	\filldraw[fill=gray!50, opacity=0.8, dashed ] (a) -- ++(d) -- (od) -- (e)-- cycle;
	\draw[very thick] (a) -- ++(t) node{$\bullet$} node[above left]{} -- ++(t2);
	
	\draw[->,Magenta,>=angle 45] (-1.5,0,0) -- (1.5,0,0);%node[below left]{$x$}
	\draw[->,Magenta,>=angle 45] (0,0,-1.8) -- (0,0,2.2);%node[above]{$z$}
	
	\filldraw[fill=gray!50, opacity=0.8, dashed ] (b) -- ++(d) -- (od) -- (e) -- cycle;
	\draw[very thick] (b) -- ++(t) node{$\bullet$} node[left]{} -- ++(t2);

	\draw[thin] (e) --  ++(d);
	\draw[Magenta] (0,-2.3,0) -- (o);

\end{tikzpicture}~\epsfxsize=160pt \epsfbox{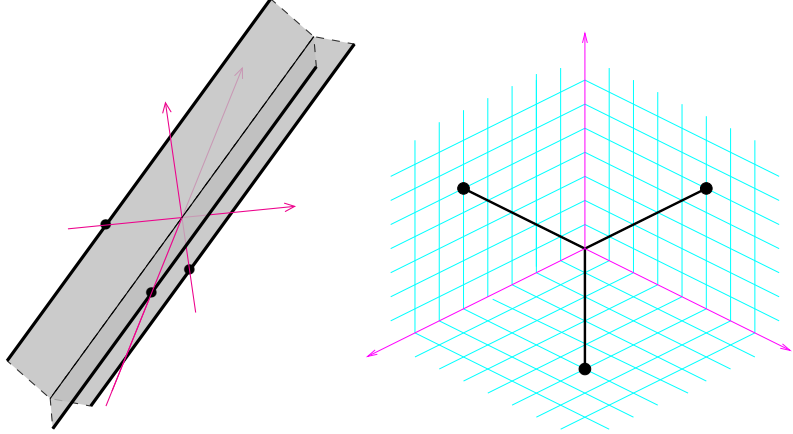} 

 \end{figure}
\end{example}

%\medskip

% \begin{remark}
% The computation of each of the above ranks for matrices over
% $(\T, \bigvee, +)$ has been proven to be $NP$-complete, cf.
% \cite{KiRo06}.
% \end{remark}

\medskip

{\em Complexity.} Computing the row, resp. column, rank of a matrix in $\T^{\nxn}$
has time-complexity~$O(n^3)$, see for instance \cite[Chapter
3.4]{Butkovic}.
Computing the tropical rank of a matrix whose entries take only two
possible values is NP-complete, see \cite[Theorem 13]{KiRo05}. On the other hand,
determining if a matrix in $\T^{\nxn}$ has tropical rank $n$
(\ie is non-singular) can be done with an algorithm of
time-complexity $O(n^3)$, see \cite{BuHe}.

\section{{\bf Max-plus spectral theory}}\label{se-spec}

The results in this section are classical, and are to be used in the proofs in
\S \ref{se-ulti}. Observe that the tropical rank, via the definition of the permanent, involves the
maximal total weight of the decompositions of $\G(A)$ into cycles.
Therefore, it should come as no surprise that it connects to the
notions below.

\begin{definition}\label{de:c(a)}
For a matrix $A\in \T^{\nxn}$, define:
\begin{equation}\label{eq:gamma}
\rho(A) = \bigvee_{ j\leq n} \bigvee_{i_1,\dots, i_j } \frac{ A_{i_1 i_2} +
  A_{i_2i_3} +\cdots + A_{i_j i_1}}{j} \:,
\end{equation}
Observe that
%$\rho(A)=-\infty$ if ${\cG}(A)$ is acyclic, and, otherwise,
$\rho(A)$ is the maximal mean weight of the (simple) circuits of
${\cG}(A)$.

A circuit of ${\cG}(A)$ is called {\em critical} if its mean weight is
$\rho(A)$.
The {\em critical graph} of $A$, denoted by ${\cG}_c(A)$, is the union of
all the critical circuits of ${\cG}(A)$.
%A matrix $A$ is called \emph{balanced} if all the scc of $\G(A)$ have a
%circuit with mean weight equal to $\rho(A)$.
\end{definition}

If ${\cG}(A)$ is acyclic then $\rho(A)=-\infty$ and ${\cG}_c(A)$ is
empty.

\begin{definition}\label{de:vpvp}
For $A \in \T^{n \times n}$, if $\lambda\in\T$ and $\bfu\in\T^n
\setminus \{(-\infty)\}$ are such that:
\[
A\odot\bfu=\lambda\odot\bf u \:,
\]
then $\lambda$ is called a \emph{(tropical) eigenvalue} and $\bf u$ is
a \emph{(tropical) eigenvector} associated to $\lambda$.
The set of such eigenvectors is the \emph{(tropical) eigenspace} associated to~$\lambda$.
%If $\bf u$ is not null, then $\lambda$ is a \emph{(tropical) eigenvalue}.
\end{definition}

Observe that the set of eigenvectors associated with a given eigenvalue is
tropically convex.

%\medskip

The next statement is a version of the celebrated max-plus spectral
theorem.

%{\bf References? Theorem 101 of greenbook ; Cunningham-Green?}
%\medskip

\begin{theorem}\label{thm:spcThm} Consider a matrix  $A \in
  \T^{\nxn}$. We have:
\begin{enumerate}
    \item\label{i:eigenvalue} $\rho(A)$ is the maximal eigenvalue of
      $A$;

    \item\label{i:weakdim} if $\rho(A)\neq -\infty$, the weak dimension of the eigenspace of $A$ associated to $\rho(A)$ is
      equal to the number of scc of $\Gc(A)$;

\item if $A$ is irreducible, or, more generally, if each scc of
      $\cG(A)$ contains a critical node, then $\rho(A)$ is the unique eigenvalue.

\end{enumerate}
\end{theorem}

A full proof can be found for instance in \cite[Chapter 3]{BCOQ}
or \cite[Chapter 4]{Butkovic}.
%\medskip

% Usually, the max-plus spectral theorem is stated as: ``if $A$ is
% irreducible, then $\rho(A)$ is the only eigenvalue of $A$,
% and~(\ref{i:weakdim}) holds''.
% But the proof of~(\ref{i:weakdim}) does not use the irreducibility and
% (\ref{i:eigenvalue}) can be checked easily or deduced from \cite{BCGa}.

\medskip

Recall that a square matrix $A$ is {\em torsion} if: $\exists k,c >0, \
A^{k+c}=A^k$.
The next result completes Theorem \ref{thm:spcThm}.

\begin{theorem}\label{th:nach}
Consider a matrix $A\in\T^{n\times n}$.
There exist $N\geq 0$, a torsion matrix $T$ of $\T^{n\times n}$, and a
sequence $(B_k)_k$ of matrices of $\T^{n\times n}$ with $\lim_k B_k =
(-\infty)$, such that:
\begin{equation}\label{eq-nach}
\forall k\geq N, \qquad A^k= k\rho(A) \odot \bigl( T^k\vee B_k \bigr)\:.
\end{equation}
Furthermore, $(-\rho(A))\odot A$ itself is torsion (that is, we can
get rid of $B_k$ in \eref{eq-nach}) iff each scc of $\G(A)$
contains a critical node. 
\end{theorem}

Theorem \ref{th:nach} is  a direct consequence of the so-called \emph{ultimate expansion}
of~\cite[Theorem 5.6]{SeSc}. The second part of the statement is an
extension of the cyclicity Theorem~\cite{CDQV85}, see for
instance \cite[Lemma 4.4]{gaub96}.

%%%%%%%%%%%%%%%%%%%%%%%%%%%%%%%%% section %%%%%%%%%%%%%%%%%%%%%%%%%%%%%
\section{{\bf The ultimate rank of a matrix}}\label{se-ulti}
\subsection{Statements}
For matrices $A$ and $B$ of compatible sizes, we
  have:
\begin{align*}
\trrk(AB) \leq \trrk(A),& \quad \trrk(AB) \leq \trrk(B)\:, \\
\rwrk(AB)\leq \rwrk(A), & \quad \clrk(AB) \leq \clrk(B)\:.
\end{align*}
It is straightforward to check these inequalities for $\clrk$ and $\rwrk$; the case of  $\trrk$
is proven in \cite[Theorem 9.4]{AGGu}.
%, and it also follows from Proposition \ref{pr:sturmfels}.

In particular, $\rk_{\star} (A^\w)$ does not increase when $\w$ increases, for $\star
\in \{ \operatorname{col},  \operatorname{row},\operatorname{tr } \}$.
It is worth looking at the limit value.
Let us start with an example.

\begin{example}\label{exmp:ranks}
Consider
\begin{equation*}\label{eq:counterex}
A =
\left[ \begin{array}{rrrr} 0 & -1 & 0 & -1 \\ -1 & 0 & -1 & 0 \\
    -\infty & -\infty & -1 & -1 \\ -\infty & -\infty & -1 &
    -1 \end{array}\right] \quad \implies \quad A^\w = \left[ \begin{array}{rrrr} 0 & -1 & 0 & -1 \\ -1 & 0 & -1 & 0 \\
    -\infty & -\infty & -\w & -\w \\ -\infty & -\infty & -\w &
    -\w \end{array}\right]
\:,
\end{equation*}
for which we have $\limw \clrk(A^\w) = \clrk(A) = 4$ while
$\limw \trrk(A^\w) = \rk_\mathrm{tr}(A) = 3$.

Now, observe that the sequence $(A^{\w})_{\w\rightarrow +\infty}$ converges to the matrix
\[
B =
\left[ \begin{array}{rrrr} 0 & -1 & 0 & -1 \\ -1 & 0 & -1 & 0 \\
    -\infty & -\infty & -\infty & -\infty \\ -\infty & -\infty & -\infty &
    -\infty \end{array}\right]
\:.
\]
For the limit matrix, the different ranks coincide: $\clrk(B) =
\rwrk(B) =\trrk(B) = 2$.
\end{example}

%Coincident of rank notions is a general property that we
%now expose.

%\medskip

%\medskip

%The terminology {\em balanced} is non-standard. Observe that an
%irreducible matrix is balanced.

\medskip

The projective
semigroup generated by a matrix  $A\in \T^{\nxn}$ is denoted
by $$\proj{A} = \{ \pi(A^\w), \ \w
\in \N\},$$ and its closure is denoted by
$\closproj{A}$. For instance, in Example \ref{exmp:ranks},
we have $\closproj{A} =\{\pi(A^k), \ k\in \N\} \cup \{\pi(B)\}$.

%\medskip

\begin{theorem}\label{th:OneGenerator}
Consider a matrix $A\in \T^{\nxn}$. For any
notion of rank (denoted $\star$) of Definition~\ref{de:ranks},
the value of $\min_{B \in \closproj{A}}\rk_{\star}
(B)$ is the same.
This common value is called the {\em ultimate rank}
of $A$ and is denoted by $ \urk(A)$.
Furthermore, we have:
\begin{equation}\label{eq:formula}
\urk(A) = \sum_{C\in  {\mathfrak C}} \cyc(C) \:,
\end{equation}
where
${\mathfrak C}$ is the set of scc of $\Gc(A)$ and
$\cyc(C)$ stands for the cyclicity of $C$.
\end{theorem}

%\medskip

Theorem \ref{th:OneGenerator} is original. However, related results appear in the litterature. In \cite[25.7: Fact
4]{AGGu} it is stated  that all the ranks coincide for a matrix $A$
which is von Neumann regular 
(\ie such that $\exists X, \ AXA=A$). A proof of this
statement appears in \cite[Corollary 1.3]{IJK}.
In \cite[Lemma 3.3]{merl10} it is proven that $\trrk(A)$ is the number of scc of
$\Gc(A)$ if $A\in \R^{\nxn}$ and $A$ is idempotent (\ie $A^2=A$).
% and in~\cite{} ({\bf ???}), it is proved that all the ranks coincide if~$A^2=A$.
%Indeed, we will first prove Theorem~\ref{th:OneGenerator} for idempotent matrices (Lemma~\ref{le:idem}).

\medskip

If $A$ is irreducible or, more generally, if each scc of $\cG(A)$
contains a critical node, then the projective semigroup
$\proj{A}$ is finite (cf. Theorem \ref{th:nach}). In this case, $\closproj{A}= \proj{A}$, and we obtain:
\begin{equation}\label{eq:only-irred}
\urk(A) = \limw \ \rk_{\star} (A^\w)\:,
\end{equation}
for any notion of rank $\star$. The equality \eref{eq:only-irred} is
not true in general as emphasized by Example \ref{exmp:ranks}.

\medskip

Theorem \ref{th:OneGenerator} will be proved in Section~\ref{sse:ProofOneGenerator}.
Before this, we give two corollaries and one example to illustrate the result.

%\medskip

\begin{corollary}\label{co:ultmax}
For a matrix $A\in \T^{\nxn}$, the following statements are
equivalent:
\begin{enumerate}
\item $A$ has maximal ultimate rank, that is, $\urk(A)=n$;
\item $\Gc(A)$ is the disjoint union of simple circuits covering
  all the nodes $\{1,\dots ,n\}$;
\item $\per(A)$ has a unique maximizing permutation and $\Gc(A)$ is
  the graph of this permutation.
\end{enumerate}
Furthermore, if the above hold, then $\per(A) = n\rho(A)$.
\end{corollary}

%\medskip

A weak version of Corollary \ref{co:ultmax}
appears in \cite[Corollary 3.5]{merl10}.

\begin{proof}[Proof of Corollary \ref{co:ultmax}]
According to \eref{eq:formula}, the ultimate rank is $n$ if and only if
$\sum_{C\in {\mathfrak C}} \cyc(C)=n$, which happens if and only if
$\Gc(A)$ is a union of disjoint simple circuits covering all the
nodes $\{1,\dots , n\}$.
Hence, $\Gc(A)$ is the graph of a permutation $\tau$ of $\{1,\cdots,n\}$.

For any  permutation~$\sigma$ of $\{1,\cdots,n\}$, $w(\sigma)=A_{1\sigma(1)}+\ldots+A_{n\sigma(n)}$ is the total weight of the edges following this permutation.
Since $\sigma$ can be decomposed into disjoint simple circuits $\sigma_1,\ldots,\sigma_k$ of lengths $\ell_1,\ldots,\ell_k$ and mean weights $w_1,\ldots,w_k$, we have
$w(\sigma)=\ell_1w_1+\dots+\ell_kw_k\le(\ell_1+\dots+\ell_k)\rho(A)=n\rho(A)$,
with equality if and only if all of the circuits of $\sigma$ are
critical circuits of $\cG(A)$.
Consequently, $\per(A)=n\rho(A)$ and $\tau=\tau_A$.
\end{proof}

\begin{corollary}\label{co:pol1} The ultimate rank of $A\in \T^{\nxn}$
can be computed with an algorithm of time-complexity $O(n^3)$.
\end{corollary}

\begin{proof}[Proof of Corollary \ref{co:pol1}]
The value of $\rho(A)$ can be computed using Karp's formula whose
time-complexity is $O(n^3)$, see for instance \cite[Theorem
2.19]{BCOQ}. The critical graph $\Gc(A)$ can also be computed with
time-complexity $O(n^3)$, see for instance \cite[Section 25.3 - Fact 13]{ABGa}.
The cyclicity
can be computed using Denardo's algorithm whose time complexity is
again $O(n^3)$, see \cite{dena}.
\end{proof}

%\medskip

%\medskip

\begin{example}
Consider the matrix $A$ defined in Example \ref{ex:1}.
Recall that $\clrk(A)= \rwrk(A)=3$ and
$\trrk(A)=2$. The critical graph of $A$ is strongly
connected and has cyclicity 1.
Applying Formula \eref{eq:formula}, we get
$\urk(A)=1$. Let us check this directly. We have:
\[
\forall \w\geq 2, \quad A^\w = A^2 =  \left[ \begin{array}{rrr} 0 & 0 & 0 \\ 0 & 0 & 0  \\ -1 & -1 & -1
\end{array}\right] \:.
\]
Therefore, $\urk(A) = \clrk(A^2)=
\rwrk(A^2)=\trrk(A^2)=1$.
\end{example}

\subsection{Proof of Theorem \protect{\ref{th:OneGenerator}}.}\label{sse:ProofOneGenerator}

We first prove the result for a matrix $E$ which is idempotent, that is,
satisfies $E^2=E$. In this case, $\closproj{E}= \{\pi(E)\}$. So we just have to
prove that: $\rwrk(E)=\clrk(E)=\trrk(E)$ equals the sum of the
cyclicities of the scc of $\cG_c(E)$. It is a direct corollary of the next
lemma.

\begin{lemma}\label{le:idem}
Let $E\in\T^{n\times n}$ be idempotent (\ie $E^2=E$) and let $r$ denote the number of scc of $\Gc(E)$:
\begin{enumerate}
\item The only eigenvalue of~$E$ belonging to $\R$ is $0$; the associated eigenspace is
the convex hull of the columns of $E$.
The column rank, resp. row rank, of $E$ is~$r$.
\item In~$\cG_c(E)$, each node holds a loop. The tropical rank of $E$ is~$r$.
\end{enumerate}
\end{lemma}

\begin{proof}%[Proof of Theorem \ref{th:OneGenerator} (the balanced case)]
%Assume that $\rho(A)=0$. (If not, carry out the proof on the normalized matrix
%$\widetilde{A}$ defined by: $\widetilde{A}_{ij} = A_{ij}-\rho(A)$.)
%According to Theorem \ref{thm:spcThm}.(3), there exists $P=A^{\w_0}$ such that
%$P^2=P$. Observe that
%\[
%\trrk(P)= \limw \trrk(A^\w), \quad \clrk(P)= \limw \clrk(A^\w), \quad
%\rwrk(P)= \limw \rwrk(A^\w) \:.
%\]
(1). Let $K$ be the convex hull of the
  columns of $E$ and let $\widetilde{K}$ be the
  eigenspace associated to 0.
Since $E^2=E$, we get that  $E\odot E_{\cdot \, j}=E_{\cdot \,
    j}$ for any column~$E_{\cdot \, j}$. Thus $K \subset
  \widetilde{K}$. Conversely, if $\bfu = (u_i)_i$ is an eigenvector of $E$ associated to some eigenvalue $\lambda\in\R$, then $$\bfu =-\lambda\odot E\odot \bfu
= \bigvee_j (u_j-\lambda) \odot E_{\cdot \, j}.$$ Therefore,  $\bfu
\in K$, which implies both that $\lambda=0$ and that $\widetilde{K} \subset
K$. So $K=\widetilde{K}$. But the weak dimension of $K$ is the
  column rank of $E$ (definition) and  the weak dimension of
  $\widetilde{K}$ is the number of scc of $\cG_c(E)$ (Theorem
  \ref{thm:spcThm}-(\ref{i:weakdim})). So $\clrk(E)=r$. We obtain
  similarly that $\rwrk(E)=r$.

\medskip

(2).
If $E$ is null, the statement is true. Otherwise, according to
  the above, $\rho(E)=0$. Assume we are in this situation.
%Now let us focus on the tropical rank. Set $k= \sum_{C \in {\mathfrak C}} \cyc(C)$, the number of scc of $\Gc(E)$.
Let~$i$ be a node of~$\cG_c(E)$. By definition, there is a critical
circuit from $i$ to $i$ of length $\ell\ge 1$, that is, $E^{\ell}_{ii} =
\ell  \rho(E)$. Since $E^{\ell}=E$ and $\rho(E)=0$, we deduce that
$E_{ii} =0$: there is a loop around $i$ in $\cG_c(E)$.

Now, consider a submatrix $Q$ of $E$ obtained by picking up exactly
one index in each scc of $\Gc(E)$ and by restricting the rows and
columns to this set of indices.
By construction, the maximum mean weight of circuits is~$0$ in~$E$ and in~$Q$ and $\per(Q)=\sum_i Q_{ii}=0$ is attained by the identity permutation.
Assume that there exists another permutation realizing the permanent,
say mapping node $i$ to $j\ne i$.
Then there is a circuit of weight $0$ containing both $i$ and $j$,
and this circuit was already present in $\Gc(E)$, which contradicts
the fact that we selected one index per scc.  Thus, the matrix $Q$ is
non-singular, and $\trrk(E)\geq r$.
But, according to \eref{eq:relation1}, $\trrk(E )\leq \clrk(E) = r$, and we conclude that $\trrk(E)=r$.
\end{proof}
%\medskip

%{\bf [To be rewritten using that the matrices $S_i$ are balanced.]}
We now turn to the proof of Theorem~\ref{th:OneGenerator} in the
general case.
If $\G(A)$ is acyclic then Theorem \ref{th:OneGenerator} is clearly
true. So we assume that $\G(A)$ is not acyclic, or, equivalently,
that $\rho(A)\neq -\infty$.
Observe that $(-\rho(A))\odot A$ has the same ranks, critical graph,
and cyclicity as $A$, but has spectral radius~0. Thus, without
loss of generality, we assume that $\rho(A)=0$.

\medskip

Set
$d =\sum_{C\in  {\mathfrak C}} \cyc(C)$,
where
${\mathfrak C}$ is the set of scc of $\Gc(A)$ and
$\cyc(C)$ stands for the cyclicity of $C$ (right-hand side of
(\ref{eq:formula})).
Since the tropical rank is the minimal one and the column and row
ranks are the maximal ones, it is enough to
prove that:
\begin{enumerate}[(i)]
\item all matrices in $\closproj{A}$ have tropical rank greater
  or equal to~$d$;
\item there is a matrix~$P\in \closproj{A}$ such that~$\clrk(P)=\rwrk(P)=d$.
\end{enumerate}
Applying Theorem \ref{th:nach}, we get that, for~$k$ large enough,
$A^k=T^k\vee B_k$.
Since $T$ is torsion, the finite entries of $T^k$ are uniformly bounded, while the entries of $B_k$ tend to~$-\infty$.
So, for $k$ large enough, we have:
\[
(T^k)_{ij}\neq -\infty \quad \implies \quad (A^k)_{ij}= (T^k)_{ij} \:.
\]
We deduce that a non-singular  submatrix of $T^k$ corresponds to a
non-singular submatrix of $A^k$, for $k$ large enough.
This proves that $\trrk(A^k)\ge\trrk(T^k)$, for $k$ large enough.
Let $T_i, i\in I$, be the matrices in the periodic part of the
ultimately periodic sequence $(T^k)_k$.
Observe that $\closproj{A} = \{\pi(A^k), k\in \N\} \cup \{\pi(T_i),
i\in I\}$. So we have:
\[
 \min_{B \in \closproj{A}} \trrk(B)= \min_k \trrk(T^k) = \min_{i\in I}
\trrk(T_i) \:.
\]
Since the sequence $\trrk(T^k)$ is non-increasing, we deduce that
$\trrk(T_i)$ is the same for all $i$.
Observe that there exists a matrix $T_j,j\in I,$ which is idempotent.
By Lemma~\ref{le:idem}, $\trrk(T_j)$ is equal to the number
of scc of $\Gc(T_j)$. 
%  and each
% node of $\cG(T_i)$ has a loop. Hence the cyclicity of every scc of
% $\cG(T_i)$ is 1. So we have that $\trrk(T_i)$ is equal to the sum of
% the cyclicities of the scc of $\cG(T_i)$.
Now, the same argument as above shows that
$\Gc(T^k)=\Gc(A^k)$, for $k$ large enough. So $\Gc(T_j)=\Gc(A^{\ell})$ for
some $\ell$, and $\trrk(T_j)$ is equal to the number
of scc of $\Gc(A^{\ell})$. Let $B$ be the boolean adjacency matrix of
$\Gc(A)$. Then  $B^{\ell}$ is the boolean adjacency matrix of
$\Gc(A^{\ell})=\Gc(T_j)$. In particular $B^{\ell}$ is idempotent, and the theory
of boolean matrices (see e.g.~\cite[Chapter 3.4]{BrRy})  tells us
that the scc of $\cG(B^{\ell})=\Gc(A^{\ell})$ are exactly the cyclicity classes
of~$\cG(B)=\Gc(A)$. We conclude that the number of scc of $\Gc(A^{\ell})$
is $d$. 
So we have $\trrk(T_j)=d$.

To prove~$(ii)$, we notice that, according to Lemma
\ref{le:idem}, $\clrk(T_j)=\rwrk(T_j)=\trrk(T_j)=d$. It completes the proof.

\subsection{Visualization and ultimate rank}\label{sse-visu}

Visualization is a standard notion already appearing
for instance in
\cite{cuni79}, and recently developped in~\cite{SSBu}. We prove a
result on the visualization of matrices with maximal ultimate rank,
Theorem \ref{th:visual}-(4), that plays an important role in Section \ref{se-ultimsemi}.

\medskip

For any finite vector ${\bf u}= (u_1, \dots, u_n)\in\R^n$, define
$\diag({\bf u})\in \T^{n\times n}$ by:
\begin{equation}\label{eq:diag}
 \diag({\bfu})_{ij}=
\left\{ \begin{array}{cc}
 u_i & \text{if } i =  j, \\[1mm]
  - \infty  & \text{if } i\neq j .
\end{array} \right.
\end{equation}

For $A\in \T^{\nxn}$, set
$\tlA =\diag({\bf -u}) A  \diag({\bf
  u})$,
where ${\bf -u}= (-u_i)_i$. The entries of $ \tlA$
satisfy:
\begin{equation}\label{eq:tlA}
\tlA_{ij}=A_{ij}+u_j-u_i.
\end{equation}
% Geometrically, $ \tlA \odot {\bf x}$ is the translation of $A \odot
% {\bf x}$ by $-\bf u$.
Matrices $A$ and $\tlA$ share many properties.

\begin{lemma}\label{le:trivial}
Given a matrix $A\in\T^{\nxn}$ and a finite vector~${\bf u}\in\R^n$,
set $\tlA =\diag({\bf -u}) A  \diag({\bf u}),$ then
$\rho(\tlA)=\rho(A)$, $\Gc(\tlA) = \Gc(A)$, $\per(\tlA)=\per(A)$,
and $\rk_{\star}(\tlA)=\rk_{\star}(A)$ for all the ranks of Definition~\ref{de:ranks}.
\end{lemma}

The proof of the lemma is straightforward.

\begin{definition}
A matrix $A\in \T^{\nxn}$ is said to be {\em visualized} (resp. \emph{strictly visualized}) if:
\begin{equation*}
\begin{array}{ll}
 A_{ij}=\rho(A)   & \text{for all} \ (i, j) \in \Gc(A), \\[1mm]
A_{ij} \leq \rho(A)\text{ (resp. }<\rho(A)\text)    &  \text{for all} \ (i, j) \notin \Gc(A).
\end{array}
%\forall [i\rightarrow j]
%\in \Gc(A), \ A_{ij}=\rho(A) \text{ and } \forall [i\rightarrow j]
%\not\in \Gc(A), \ A_{ij} \leq \rho(A)\text{ (resp. }<\rho(A)\text) \:. \label{eq:visual}
% \\
% \forall [i\rightarrow j] \in \cC(A), \ A_{ij}=\gamma(A), &  \forall [i\rightarrow j]
% \not\in \cC(A), \ A_{ij} < \gamma(A) & \label{LM:eq:visual2}
\end{equation*}
A finite vector ${\bf u}=(u_i)_i \in\R^n$ is called a {\em (strict)
  visualization}  of~$A$ if the matrix $ \diag({\bf -u}) A  \diag({\bf
  u})$ is (strictly) visualized.
\end{definition}

We now define the fundamental cell introduced in a
different form in \cite{butk00}.
%and under a different name (``simple image set'') in \cite{butk00}.

\begin{definition}\label{de:fcell}
The {\em fundamental cell} of a non-singular matrix $A \in\T^{\nxn}$ is defined as the set
\[\fcell(A) =
\left\{ {\bf x} \in \R^n \mid \forall i, \forall j \neq \tau_A(i),
  \quad
  A_{ij} + x_j < A_{i\tau_A(i)} + x_{\tau_A(i)} = (A \odot \bfx)_i \right\}\:.
\]
The {\em fundamental cell} of a singular matrix $A$ is empty, $\fcell(A)=\emptyset$.
\end{definition}

Let us mention several properties of the fundamental cell which can
be obtained by adapting the results from \cite{butk00} and
\cite{merl10}. 

Let $A$ be non-singular. The fundamental cell $\fcell(A)$ is non-empty
and its topological dimension is $n$ since it is an open set.
Consider the mapping  $\varphi_A: \R^n \rightarrow
\R^n, \ {\bf x} \mapsto A\odot {\bf x}$. 
We have
\[
\fcell(A) = \varphi_A^{-1}(\bigl\{ {\bf x} \in \R^n, \ \exists !
{\bf y} \in
\R^n, {\bf x}=\varphi_A({\bf y)} \bigr\})\:.
\]
(The set $\bigl\{ {\bf x} \in \R^n, \ \exists !
{\bf y} \in
\R^n, {\bf x}=\varphi_A({\bf y)} \bigr\}$ is defined for any
matrix $A$, and is non-empty if and only if $A$ is non-singular.)
The restriction of
$\varphi_A$ to the domain $\fcell(A)$ is the
affine map given by:
$$
\begin{array}{ll}
\varphi_A: \fcell(A)     \longrightarrow  \R^n, &   \\[1mm]
\qquad \ {\bf x}  \longmapsto {\bf y},    \ y_i = A_{i\tau_A(i)} +
x_{\tau_A(i)} \:. &
\end{array}$$
In particular, on the domain $\fcell(A)$, the map $\varphi_A$  is an isometry
for the euclidean distance.

% The "simple image set"
% studied in~\cite{butk00} is the image by~$\varphi_A$ of the fundamental
% cell.

% (in the usual
% sense), and therefore an isometry for the euclidean distance.

% convex in the usual sense {\bf and tropically convex ?? If someone wrote that, please give reference}.
% It is the maximal open set on which the map $f_A:x\mapsto A\odot x$ coincide with a reversible affine map, nd thus introduces an affine isometry \cite[Lemma 3.9]{merl04}.
% It is proved in~\cite[Theorems~3.3 and~4.1]{butk00} that such a reversible map exist if and only if $A$~is non-singular.
% This is the first step of an inductive proof that the tropical rank
% of~$A$ is the topological dimension of the image of~$f_A$.
% The simple image set
% studied in~\cite{butk00} is the image by~$f_A$ of the fundamental
% cell.

\medskip

%It is folklore, that any eigenvector of~$A$ is a visualization of~$A$.
The next theorem  summarizes the results that we need about
visualizations.
%Obviously, if $\rho(A)=-\infty$, there is no visualization, except for the null matrix.

\begin{theorem}\label{th:visual}
Let~$A\in \T^{n\times n}$ be a matrix with $\rho(A)\neq -\infty$.
\begin{enumerate}
 \item The visualizations of $A$ are the vectors ${\bf u}\in
   \R^n$ such that $A\odot {\bf u} \leq \rho(A)\odot {\bf u}$.
 \item There exists a strict visualization of $A$.
\end{enumerate}
Assume that $\Gc(A)$ contains all the nodes $\{1,\dots ,
n\}$.
\begin{enumerate}
 \item[(3)] The visualizations of~$A$ are the eigenvectors 
   (associated to $\rho(A)$).
\end{enumerate}
Assume that $\Gc(A)$ is the disjoint union of circuits covering all
the nodes $\{1,\dots , n\}$ (equivalently, that $\urk(A)=n$).
\begin{enumerate}
\item[(4)] The strict visualizations of~$A$ are the eigenvectors (associated to $\rho(A)$)
  belonging to the fundamental cell.
\end{enumerate}
\end{theorem}

\begin{proof}
Property (1) is folklore. Let us prove
it for completeness sake. First, the vector ${\bf u}\in
   \R^n$ is a visualization of $A$ if and only if
   \begin{equation}\label{eq:vis} \forall i,j, \qquad A_{ij}+u_j -u_i
   \leq \rho(A) \:.\end{equation} 
Indeed, assume \eref{eq:vis} holds and
consider a critical cycle $(i_1,\dots , i_k,i_{k+1}=i_1)$ of $\cG(A)$, then:
\[
\sum_{\ell =1}^{k} A_{i_{\ell}i_{\ell+1}}+u_{i_{\ell+1}}
-u_{i_{\ell}} = \sum_{\ell =1}^{k} A_{i_{\ell}i_{\ell+1}} =
k\rho(A)\:.
\]
Thus, \eqref{eq:vis} implies that $ A_{i_{\ell}i_{\ell+1}}+u_{i_{\ell+1}}
-u_{i_{\ell}} =\rho(A)$ for all $\ell$.

Second, we have the following equivalences:
\begin{eqnarray*}
\bigl[ \forall i,j, \ A_{ij}+u_j -u_i
   \leq \rho(A) \bigr] & \iff & \bigl[ \forall i, \ \max_j (A_{ij}+u_j)
   \leq \rho(A) +u_i \bigr] \\
& \iff & \bigl[ A \odot {\bf u} \leq \rho(A)\odot {\bf u} \bigr] \:.
\end{eqnarray*}
This completes the proof of property (1).

Property (2) is proven in~\cite[Proposition~3.4]{SSBu}.

To obtain
property $(3)$, it
suffices to show that $[A\odot {\bf u} \leq \rho(A)\odot {\bf
  u}]\implies [A\odot {\bf u} = \rho(A)\odot {\bf
  u}]$. Assume that  ${\bf u}\in
   \R^n$ satisfies $A\odot {\bf u} \leq \rho(A)\odot {\bf
  u}$, and fix an arbitrary node~$i$. Since $\Gc(A)$ contains all the nodes, there
  exists a critical arc $(i , j)$.  By property (1), the vector ${\bf u}$ is a
visualization of $A$. In particular,
we have  $A_{ij} + u_{j}-
u_{i} = \rho(A)$, see \eref{eq:tlA}. Hence, $(A\odot {\bf u})_i \geq A_{ij}+u_j = \rho(A) +
u_i$. So we have proved that  $A\odot {\bf u} \geq \rho(A)\odot {\bf
  u}$. We conclude that $A\odot {\bf u} = \rho(A)\odot {\bf
  u}$.

Let us show property $(4)$. Assume that $\urk(A)=n$, then, the
visualizations of~$A$  are its eigenvectors by property (3). Let ${\bf
  u}$ be an eigenvector. Then 
${\bf u}$ is a visualization which is strict if and only if $A_{ij} + u_j -u_i <
\rho(A)$ for $(i , j)\not\in \Gc(A)$.
By Corollary \ref{co:ultmax}, we
have $$[(i , j)\not\in \Gc(A)] \ \iff \ [j \neq \tau_A(i)],$$ where $\tau_A$ is
the unique maximizing permutation of $A$. Therefore,  the
visualization~${\bf u}$ is strict if and only if
\[
\forall i, \forall j \neq \tau_A(i), \qquad A_{ij} + u_j <
\rho(A) + u_i = (A\odot {\bf u})_i = A_{i\tau_A(i)} + u_{\tau_A(i)} \:,
\]
which is equivalent to saying that ${\bf u}$ belongs to $\cF(A)$.
\end{proof}

%%%%%%%%%%%%%%%%%%%%%%%%%%%%%%%%% section %%%%%%%%%%%%%%%%%%%%%%%%%%%%%
\section{{\bf The ultimate rank of a semigroup of matrices}}\label{se-ultimsemi}
\subsection{Statements}\label{sse-Statements}

Let us extend the notion of ultimate rank to a semigroup of tropical
matrices.

\begin{definition}
Let $\SG \subset \T^{\nxn} $ be a semigroup of tropical
matrices. The \emph{ultimate rank} of $\SG$, denoted $ \urk(\SG)$, is
defined by:
%$$ \urk(\SG) := \min \{ \urk (A) \ds | A \in \SG) \}\:.$$
$$ \urk(\SG) = \min \{ \rk_{\star} (A) \ds | A \in
\overline{\pi(\SG)} \}= \min \{ \urk (A) \ds | A \in \overline{\pi(\SG)} \}.$$
\end{definition}
%Observe that obviously $\urk(\SG) = \min \{ \urk (A) \ds | A \in \overline{\pi(\SG)} \}$.

We do not
know if the ultimate rank of a given finitely generated semigroup is
algorithmically computable. However, we prove a partial result.
Indeed, we give two characterizations of the case of maximal ultimate rank,
one combinatorial and one geometric. As a by-product, we
obtain a polynomial-time algorithm to decide if the ultimate rank is
maximal.

\medskip

\begin{lemma}\label{pr:burnside}
Let ${\SG}$ be a finitely generated semigroup of matrices of
$\T^{\nxn}$. Assume that $\forall P \in \SG, \urk(P)=n$. Then $\pi({\SG})$ is finite.
\end{lemma}

\begin{proof}
In \cite[Theorem 2.1]{gaub96}, it is proven that a finitely generated torsion semigroup of $\P\T^{\nxn}$ is finite
(\ie the Burnside problem has a positive answer in $\P\T^{\nxn}$).
Therefore,  we only have to prove that $\pi({\SG})$ is torsion, \ie for any $P\in\SG$, there exist
$\w,c >0$, such that $\pi(P^{\w+c})=\pi(P^\w)$.

Since $\urk(P)=n$, by Corollary~\ref{co:ultmax}, all the nodes of
$\cG(P)$ are critical. Then, by Theorem \ref{th:nach}, the matrix
$(-\rho(P))\odot P$ is torsion, or, equivalently, $P$ is projectively
torsion.
\end{proof}

Lemma \ref{pr:burnside} enables to obtain a simplified
characterization of finitely generated semigroups of maximal ultimate
rank.

\begin{lemma}\label{pr:def}
Let ${\SG}$ be a finitely  generated semigroup of matrices of
$\T^{\nxn}$. Then
\[
\bigl[ \urk(\SG) = n \bigr] \ \Longleftrightarrow \ \bigl[ \forall P \in \SG, \urk(P) = n
\bigr]\:.
\]
\end{lemma}

\begin{proof}
Obviously, we have $\bigl[ \urk(\SG) = n \bigr] \implies \bigl[
\forall P \in \SG, \urk(P) = n\bigr]$. Conversely, assume that $\forall P \in
\SG, \urk(P) = n$. By Lemma \ref{pr:burnside}, $\pi({\SG})$ is
finite, and, in particular, $\pi({\SG})=\overline{\pi(\SG)}$. We
conclude that $\urk(\SG) = n$.
\end{proof}

%\medskip
% JM: IF WE WANT SUCH A STATEMENT, WE NEED TO BE PRECISE ... NO "MOST
% OF OUR RESULTS" !! ALSO WE ADD IN THE END IN AN "EXTENSION"
% SUBSECTION. IT IS CLEARER FOR THE READER.
%

Let us state the main results.

\begin{theorem}[Combinatorial characterization]\label{Thm:Comb}
Let $\SG=\langle\gen\rangle$ be the semigroup generated
by a finite set $\gen$ of matrices in~$\T^{n\times n}$.

For $A\in \gen$, define $\widetilde{A} \in \T^{n\times n}$ by: $\forall i,j, \ \widetilde{A}_{ij} = A_{ij}
- \rho(A)$.
Set $M=\bigvee_{A\in\gen} \widetilde{A}$. We have $\urk(\SG)=n$ if and
only if the following three properties are satisfied:

\begin{description}
\item[(C1)]\label{i:cond3.1} $\forall A \in\gen, \ \urk(A) = n$;
 \item[(C2)]\label{i:cond3.2} $\rho(M)=0$;
 \item[(C3)]\label{i:cond3.3} $\forall A \in\gen, \forall i,j \in \{1,\dots,
   n\}$, \begin{equation*}\Bigl[
   (i, j) \in \Gc(M), \widetilde A_{ij}=M_{ij} \Bigr]
   \implies  (i,j) \in \Gc(A). \end{equation*}
\end{description}
% \item\label{i:cond2} the following two properties are satisfied:
% \begin{enumerate}[(i)]
% \item\label{i:cond2.1} $\forall A \in\gen,\urk(A) = n$;
% \item[(iv)]\label{i:strviz} there exist a common strict visualization of all $A\in\gen$.
% \end{enumerate}
% \end{enumerate}
% If~$\gen$ is finite, then all those properties are equivalent to $\urk(\SG)=n$
\end{theorem}

\begin{theorem}[Geometric characterization]\label{th:geom}
Let $\SG=\langle\gen\rangle$ be the semigroup generated
by a finite set $\gen$ of matrices in~$\T^{n\times n}$.  We have $\urk(\SG)=n$ if and
only if the following property holds:

\begin{description}
\item[(G)] the generators have a common finite
eigenvector that belongs to the intersection of their fundamental
cells;
\end{description}
or, equivalently, if the following two properties hold:
\begin{description}
\item[(G1)] $\forall A \in\gen, \ \urk(A) = n$;
\item[(G2)] the generators have a common strict visualization.
\end{description}
\end{theorem}

The proofs of Theorems \ref{Thm:Comb} and \ref{th:geom} are given in \S\ref{sse-Proofs}.
To the best of our knowledge, Theorem \ref{Thm:Comb} is completely original.
Theorem \ref{th:geom} refines \cite[Theorem 3.1]{merl10} which states
that: $[\urk(\SG)=n] \implies$ the matrices in the semigroup have a
common eigenvector. The proof of \cite[Theorem 3.1]{merl10} is
different from our proof and relies on Kakutani fixed point Theorem.
%  \eqref{i:cond1} imples the existence of a common eigenvector, based on Kakutani fixed point Theorem, that is not used in this paper.
% One one hand, \cite[Theorem 3.1]{merl10} did not requires the generating set to be  projectively bounded,
% on the other hand it requires at least one matrix in $\SG$ to have a projectively bounded image (i.e. all column vectors are either null or finite).

\begin{corollary}\label{co:pol2}
Let $\SG=\langle\gen\rangle$ be the semigroup generated
by a finite set $\gen$ of matrices in~$\T^{n\times n}$.
There exists an algorithm of time-complexity $O(|\gen|n^3)$
that decides whether $\urk(S) =n$.
\end{corollary}

\begin{proof}
We use the characterization in Theorem~\ref{Thm:Comb}.
Property $(C1)$ can be checked in $O(|\gen|n^3)$ using Corollary \ref{co:pol1}. The
matrix $M$ can be computed in $O(|\gen|n^2)$ and
$\rho(M)$ can be computed in $O(n^3)$ using Karp's formula, see \cite[Theorem
2.19]{BCOQ}.
So property $(C2)$ can be verified in $O(|\gen|n^2 + n^3)$.
Checking property $(C3)$ requires to
compute $\Gc(M)$ and $\Gc(A)$ for every $A\in\gen$, which can be done in
$O(|\gen|n^3)$, see for instance \cite[Section 25.3 - Fact 13]{ABGa}. 
\end{proof}

\subsection{Proof of Theorem~\ref{Thm:Comb} and Theorem~\ref{th:geom}}\label{sse-Proofs}
We carry out the proofs of
Theorems~\ref{Thm:Comb} and~\ref{th:geom} together.
Let $\SG=\langle\gen\rangle$ be the semigroup generated by the finite
set $\gen$ of matrices in~$\T^{n\times n}$.
For simplicity and without loss of generality, we assume that
$\rho(A)=0$ for any $A \in \gen$. (Otherwise, we replace $A$ by
$-\rho(A)\odot A$.)
Set $M=\bigvee_{A\in\gen} A$.
Set:
\begin{description}
\item[(P1)] $\urk(\SG) = n$;
\item[(P2)] Properties $(C1)$, $(C2)$, and $(C3)$ in Theorem
  \ref{Thm:Comb};
\item[(P3)] Properties $(G1)$ and $(G2)$ in Theorem
  \ref{th:geom}.
\end{description}
The structure of the proof is:
$${\bf (P1)} \implies {\bf (P2)} \implies {\bf (P3)} \implies{\bf (P1)} \:.$$

% Below, we denote the elements of~$\gen$ by $A$, $A(k)$ or $B$, and the
% elements of $\SG$ by~$P$ or $Q$.

% Thus, for the end of this section, $\SG=\langle\gen\rangle$ will be
% the semigroup generated by a set $\gen$ of matrices in~$\T^{n\times
% n}$ and properties~\eqref{i:cond1},\eqref{i:cond2},\eqref{i:cond3}
% and~
% \eqref{i:cond3.1},\eqref{i:cond3.2},\eqref{i:cond3.2} refer to Theorems~\ref{Thm:Comb} and~\ref{th:geom}.

% To avoid confusions, we will denotes the elements of~$\gen$ by $A$, $A(k)$ or $B$, and the elements of $\SG$ by~$P$ or $Q$.

\medskip

[{\bf (P1)}  $\implies$  {\bf (P2)}].
Assume $(P1)$ holds. Then, $(C1)$ is obvious. Let us prove $(C2)$,
that is $\rho(M)=0$.
Observe that, for any $A\in \gen$, we have $M\geq A$, so that
$\rho(M)\ge\rho(A)= 0$.
Now, $(C2)$ follows from the following lemma.
\begin{lemma}\label{le:M}
Assume that $(C1)$ holds and that
$\rho(M)>0$. Then, there is a product~$P$ of at most~$n$ matrices of~$\gen$
for which $\urk(P)<n$.
\end{lemma}

\begin{proof}
Since $\rho(M)>0$, there is a simple circuit, say $(i_1,\cdots, i_k,i_{k+1}=i_1)$, with positive
weight~$w$ in $\G(M)$.
For each $\ell \in \{1, \dots, k \}$, choose $A(\ell) \in \gen$ such that
$A(\ell)_{i_{\ell}i_{\ell+1}}=M_{i_{\ell}i_{\ell+1}}$,
and  define $P=A(1) \cdots A(k)$. % (cf. Fig.~\ref{fig:rho>0}, where we
% assume~$i_l=l$ for shorter notations).
Observe that $\rho(P)\geq P_{i_1 i_1} \geq w>0$.
% \begin{figure}\label{fig:rho>0}
% \begin{VCPicture}{(-5,-0.5)(5,3.5)}
% %\SetStateLineColor{black}
% % states
% \tiny
% \State[1]{(3,1)}{1}
% \VCPut[60]{(0,0)}{\State[2]{(3,0)}{2}}
% \VCPut[120]{(0,0)}{\State[3]{(3,0)}{3}}
% \State[4]{(-3,1)}{4}
% \ArcR{1}{2}{M_{12}=A(1)_{12}}
% \ArcR{2}{3}{M_{23}=A(2)_{23}}
% \ArcR{3}{4}{M_{34}=A(3)_{34}}
% \ChgEdgeLineStyle{dotted}
% %\ArcR{4}{1}{}
% \VArcR{arcangle=-35,ncurv=1.05}{4}{1}{}
% \end{VCPicture}\caption{$\rho(P)>0$}
% \end{figure}
Assume that $\urk(P)=n$.
According to Proposition~\ref{pr:merlet} and Corollary~ \ref{co:ultmax}, we must have
$$\rho(P)=\frac{1}{n}\per(P) =
\frac{1}{n}\sum_{i} \per(A(i))= \sum_{i} \rho(A(i)) = 0\:.$$
So we have reached a contradiction, showing that $\urk(P)<n$.
\end{proof}

Now let us prove property $(C3)$. We need a preliminary lemma.

% \begin{lemma}\label{le:M2}
% We have:
% \[
%  \rho(M) =0 \ \implies \ \forall P \in \SG, \
% \rho(P) =0 \:.
% \]
% \end{lemma}

% \begin{proof}
% Consider  $P \in
% \SG$. Assume $P$ can be written as a product $A_1\cdots A_k$  in the
% generators, then we have $P\leq M^k$. Therefore $\rho(P)\leq
% \rho(M^k)=0$.
% Let $\tau_{j}$ be the unique maximizing permutation of
% $\per(A_{j})$, and set $\tau=\tau_k\circ\cdots \circ \tau_1$. We
% have:
% \[
% \forall \ell, \ P_{\ell\tau(\ell)}\geq
% (A_{1})_{\ell, \tau_1(\ell)} +
% (A_{2})_{\tau_1(\ell), \tau_2\circ\tau_1(\ell)} +\cdots +
% (A_{k})_{\tau_{k-1}\circ \cdots \circ \tau_1(\ell), \tau(\ell)} = 0 \:.
% \]
% Therefore, $\rho(P)\geq 0$, and we conclude that $\rho(P)=0$.
% \end{proof}

\begin{lemma}\label{le:visua}
Assume that $(C1)$ and $(C2)$ hold. We have:%For any matrix~$A\in\gen$, $\Gc(A)$ is a subgraph of $\Gc(M)$.
\begin{enumerate}
\item[(i)] $\forall P \in \SG, \ \rho(P)=0$;
\item[(ii)] a visualization of~$M$ is a visualization for any $P$ in~$\SG$.
 %  \item\label{i:visuaCG}
 % For any matrix~$A\in\gen$, $\Gc(A)$ is a subgraph of $\Gc(M)$.
\end{enumerate}
\end{lemma}

\begin{proof}
By Theorem \ref{th:visual}-(1), a  visualization ${\bf
  v}$  of $M$ satisfies $M\odot {\bf
  v} \leq {\bf v}$. Given a  matrix $A\in\gen$, we have
$$A\odot {\bf v} \leq M\odot {\bf v}\leq {\bf v}\:.$$
Therefore, by  Theorem \ref{th:visual}-(1) again, ${\bf v}$ is a
visualization of any generator $A$. Now any generator $A$ has maximal ultimate
rank by property $(C1)$. According to Corollary \ref{co:ultmax}, all the nodes of~
$\cG(A)$ are critical. By  Theorem \ref{th:visual}-(3), it implies that
${\bf v}$ is an eigenvector of~ $A$. Since ${\bf v}$ is a common
eigenvector of the generators, it is a common eigenvector for all the
matrices in the semigroup: $\forall P \in \SG, \ P\odot {\bf v} = {\bf
  v}$. It implies, by Theorem
\ref{th:visual}-(1), that ${\bf
  v}$ is a visualization for $P$. It also implies that $\rho(P)\geq
0$. Since $P=A_1\cdots A_{\ell}$ for some $\ell$ and $A_i\in \gen$, we
get that $P\leq M^{\ell}$, implying that $\rho(P)\leq 
\rho(M^{\ell})=0$. We conclude that $\rho(P)=0$. 
\end{proof}

From now on, we assume without loss of generality that $M$ is strictly visualized
and that all the matrices of the semigroup are visualized. This is possible for the
following reason. Recall that for any $P\in \T^{n\times n}$ and ${\bf
  u}\in \R^n$, the matrices $P$ and and $\diag(-{\bf u}) \odot P \odot \diag({\bf u})$  have
the same ranks, critical graph, etc, see Lemma \ref{le:trivial}. 
Let ${\bf v}$ be a strict visualization of $M$, which exists by
Theorem \ref{th:visual}. By Lemma \ref{le:visua}, ${\bf
  v}$ is a visualization for the matrices in the semigroup. 
Now replace $A$ by
$\diag(-{\bf v}) \odot A \odot \diag({\bf v})$ for
$A \in \gen$, and $M$ by $ \diag(-{\bf v}) \odot M \odot
\diag({\bf v})$. 

\begin{lemma}\label{le:3.3necess}
Assume that $(C1)$ and $(C2)$ hold but $(C3)$ does not.
Then, there is a product~$P$ of at
most~$n$ matrices of~$\gen$ such that $\urk(P)<n$.
\end{lemma}

\begin{proof}
% For $A\in\gen$, an arc $[i\rightarrow j] \in \Gc(A)$ is such that
% $\widehat{A}_{ij}=0$ by visualization. So $\widehat{M}_{ij}=0$, which
% implies that  $[i\rightarrow j] \in \Gc(M)$ by strict visualization.
% We have proved that $\Gc(A)$ is a subgraph of $\Gc(M)$.
Property $(C3)$ does not hold. So there exists $B \in \gen$ and
$(i_0,i_1) \in \Gc(M)$ such that: $B_{i_0i_1}=M_{i_0i_1}=0$, but $(i_0,i_1)
\not\in \Gc(B)$. 
% Since $M_{i_0i_1}=0$ and since $M$ is strictly
% visualized, we get that $(i_0,i_1)\in \Gc(M)$. 
The critical arc $(i_0,i_1)$
can be completed into a critical circuit $(i_0,i_1,i_2,\dots, i_{\ell},
i_{\ell+1}=i_0)$ of $\Gc(M)$. For each $k = 1,\dots,\ell$, we choose a matrix $A(k)\in\gen$
such that $A(k)_{i_k i_{k +1}}=M_{i_k i_{k+1}}$. Set $P=A(1) \cdots
A(\ell)$ and $\tau=\tau_{A(\ell)}\circ\cdots\circ \tau_{A(1)}$. (See
Figure~\ref{fig:cond3nec}.)

\begin{figure}
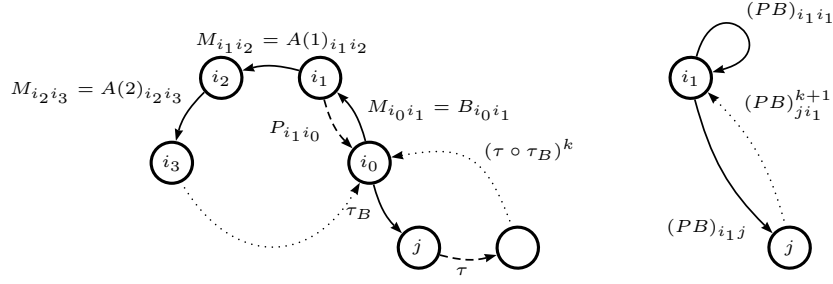


\scalebox{0.65}{
\begin{VCPicture}{(-4.5,-3)(7,4)}

%\SetStateLineColor{black}
% states
\tiny
%\Gc(M)-states
\State[i_0]{(0,0)}{a}
\VCPut[60]{(-2,0)}{\State[i_1]{(2,0)}{b}}
\VCPut[120]{(-2,0)}{\State[i_2]{(2,0)}{1}}
\VCPut[180]{(-2,0)}{\State[i_3]{(2,0)}{2}}

%\Gc(PB)-states
\VCPut[60]{(5.5,0)}{\State[i_1]{(2,0)}{bM}}
\VCPut[-120]{(9.5,0)}{\State[j]{(2,0)}{cM}}

%\ArcR{a}{b}{M_{ab}=0}
%\ArcR{b}{1}{0=M_{bx}}
%\ArcR[0.3]{1}{2}{M_{xy}}
\ArcR{a}{b}{M_{i_0i_1}=B_{i_0i_1}}
\ArcR{b}{1}{M_{i_1i_2}=A(1)_{i_1i_2}}
\ArcR[0.3]{1}{2}{M_{i_2i_3}=A(2)_{i_2i_3}}

\VCPut[-120]{(2,0)}{\State[j]{(2,0)}{c}}
\ArcR{a}{c}{\tau_B}

\ChgEdgeLineStyle{dashed}

\ArcR{b}{a}{P_{i_1i_0}}

\VCPut[-60]{(2,0)}{\State{(2,0)}{3}}
\ArcR{c}{3}{\tau}

\ChgEdgeLineStyle{dotted}
\VArcR{arcangle=-40,ncurv=1.1}{3}{a}{(\tau\circ\tau_B)^k}
\VArcR{arcangle=-60,ncurv=1.1}{2}{a}{}
\ArcR[0.75]{cM}{bM}{(PB)^{k+1}_{ji_1}}

\ChgEdgeLineStyle{solid}
\ArcR[0.85]{bM}{cM}{(PB)_{i_1j}}
\LoopNE[0.5]{bM}{(PB)_{i_1i_1}}
\end{VCPicture}
}

\caption{Illustration of the proof of Lemma \ref{le:3.3necess}}\label{fig:cond3nec}

\end{figure}

By construction, $P_{i_1i_0}=0$. We have:
\[
(PB)_{i_1i_1} \geq P_{i_1i_0} + B_{i_0i_1} = 0 \:.
\]
Since $PB$ is visualized by assumption and  $\rho(PB)=0$ by Lemma
\ref{le:visua}-(i), we get $(PB)_{i_1i_1} =0$. There is a loop
around $i_1$ in $\Gc(PB)$. Let us prove that there is another critical
circuit going through~$i_1$ in $\Gc(PB)$.

Let $j\neq i_1$ be such that $(i_0,j) \in \Gc(B)$. By visualization,
$B_{i_0j}=0$. Therefore,
\[
(PB)_{i_1j}  \geq P_{i_1i_0}  + B_{i_0j} = 0 \ \implies \ (PB)_{i_1j} =0 \:.
\]
Let $k\in \N$ be such that $(\tau_B\circ\tau)^{k+1} (j)=j$. We have:
$$
\tau\circ(\tau_B\circ\tau)^{k} (j)=\tau_B^{-1} (j)=i_0 \quad \implies
\quad 
((PB)^{k}P)_{ji_0}=0 \quad \implies \quad (PB)^{k+1}_{ji_1}=0\:.
$$
Thus, in $\Gc(PB)$, there is a circuit going from $i_1$ to $j$ and
back to $i_1$. 
Consequently, we have built a product $PB\in\SG$ of at most~$n$ matrices whose critical
graph contains two different circuits passing through a node. Then,  $\urk(PB)<n$ by Corollary~
\ref{co:ultmax}.
\end{proof}

Since {\bf (P1)} holds, property (C3) follows from Lemma \ref{le:3.3necess}. This completes
the proof.

\medskip

[{\bf (P2)} $\implies$ {\bf (P3)}].

First of all, Property~$(G1)$ is the same as Property~$(C1)$, so it follows from~$(P2)$.

To prove that~$(G2)$ holds, let us assume as above that $M$ is strictly visualized and that the matrices in $\SG$
are visualized. We are going to show that the generators are strictly visualized which will prove $(G2)$.

Consider $A \in \gen$ with
$A_{ij}=0$ for some $i,j \in \{1,\dots, n\}$. Then $M_{ij}\geq A_{ij}=0$. Thus $M_{ij}=0$, and by
strict visualization $(i, j)\in \Gc(M)$. We
conclude by $(C3)$ that  $(i, j) \in \Gc(A)$. This means precisely that $A$ is
strictly visualized.

% \begin{lemma}\label{le:3.3suffic}
% Assuming (P2)~holds, a strict visualization of~$M$ is also a strict
% visualization for any matrix in~$\SG$.
% \end{lemma}

% \begin{proof}
% Assume, as in the proof of Lemma \ref{le:3.3necess}, that $M$ is strictly visualized and the matrices in $\SG$
% are visualized. It is
% enough to prove that all matrices in $\SG$ are strictly visualized.
% Suppose $A \in \gen$ with
% $A_{ij}=0$ for some $i,j \in \{1,\dots, n\}$, then $M_{ij}\geq A_{ij}=0$. Thus $M_{ij}=0$, and by
% strict visualization $(i\rightarrow j)\in \Gc(M)$. We
% conclude by (C3) that  $(i\rightarrow j) \in \Gc(A)$. This precisely means  that $A$ is
% strictly visualized. The proof extends directly to any matrix in
% $\SG$.
% \end{proof}

\medskip

The fact that $(G1)$ plus $(G2)$ is equivalent to $(G)$ follows directly from  Theorem \ref{th:visual}-(4).

\medskip

[{\bf (P3)} $\implies$ {\bf (P1)}].

%The matrices~$A\in\gen$ being non-singular,they have a finite $\rho(A)$. Thus we can replace $A$ by~$(-\rho(A))\odot A$ and assume that $\rho(A)=0$.

Let $\bfu\in \R^n$ be a common eigenvector of all the generators
that lies in the intersection $\mathcal W = \bigcap_{A \in
  \gen}\fcell(A)$, which exists according to property $(G)$. Let us show that every $P\in\SG$, written as a product $A_1\cdots A_{\ell}$, $A_i\in \gen$, is non-singular.

Denote by  $d(\cdot\; , \cdot)$ the Euclidean distance of $\R^n$.
Since $\mathcal W$ is a non-empty intersection of finitely many open sets, there exists $\varepsilon>0$ such that the ball
${\mathfrak B} = \bigl\{ {\bf x}\in \R^n \mid d(\bfu, {\bf x}) \leq
\varepsilon \} \subset \mathcal W$. We use the notations:
$\varphi_i: {\bf x} \mapsto A_i \odot {\bf x}$,  and   $\varphi_P:
{\bf x} \mapsto P\odot {\bf x}$.
Recall that $\varphi_i$ is an affine isometry on $\fcell(A_i)$, see
\S \ref{sse-visu}. Since
$A_i\odot \bfu =\bfu$, we obtain that $\varphi_i({\mathfrak B}) =
{\mathfrak B}$.  By composition,  we get that $\varphi_P({\mathfrak B}) =
{\mathfrak B}$. In particular, ${\mathfrak B}$ is included in the
image of $\varphi_P$, or, equivalently, in the tropical convex hull of the
columns of $P$. Since ${\mathfrak B}$ is of topological dimension~$n$,
we get that  the tropical convex hull of the
columns of $P$ is of topological dimension $n$. Using Proposition
\ref{pr-sturmfels}, we conclude that $P$ is non-singular.

Therefore, we have proved that all the matrices in $\SG$ have tropical
rank $n$. In particular, given a matrix $P$ in $\SG$, all the products
$P^k$ have tropical rank $n$, which implies that the ultimate rank of
$P$ is $n$. According to Lemma \ref{pr:def}, this implies that the ultimate
rank of $\SG$ is $n$.

\subsection{Examples}

We illustrate Theorems \ref{Thm:Comb} and \ref{th:geom} using four
successive examples. All the matrices in these examples belong to
$\T^{3\times 3}$ and are non-singular. Their projective fundamental
cells belong to $\P\R^3$, and are
represented (in $\R^2$) by orthogonal projection on the
plane orthogonal to the direction $(1,1,1)$.

Here is a general observation that is useful for the examples below. For $\bfu \in \R^n$, define the matrix $\diag(\bfu)\in \T^{\nxn}$ like in~(\ref{eq:diag}).
The fundamental cell of $\diag({\bf -u}) A  \diag({\bf
  u})$ is the translation of the
fundamental cell of $A$ by~${\bf -u}$.
% Given a matrix $A\in \T^{\nxn}$, set $\tlA =\diag({\bf -u}) A  \diag({\bf
%   u})$. Then, the fundamental cell of $\tlA$ is the translation of the
% fundamental cell of $A$ by~$\bfu$.

\begin{example}
Consider the matrices:
\[
A_1 = \left[ \begin{array}{rrr} 0 & -2 & -2 \\ -2 & 0 & -2 \\ -2 & -2 &
    0 \end{array} \right]\:, \qquad A_2 = \left[ \begin{array}{rrr} 0 & -5 & -2 \\ 1 & 0 & 1 \\ -2 & -5 &
    0 \end{array} \right] \:,
\]
where $A_2=\diag({\bf -u})\odot A_1 \odot \diag({\bf
  u})$ for $\bfu= (0,-3,0)$. We have $\urk(A_1)=\urk(A_2)=3$.
For each matrix, the fundamental cell
is the interior of the  set of eigenvectors.
Here the
intersection of the fundamental cells is non-empty, see Figure
\ref{fi:exam1}. Applying Theorem \ref{th:geom}, we conclude that $\urk\langle A_1,A_2 \rangle = 3$.

\begin{figure}[ht]
\[ \epsfxsize=140pt \epsfbox{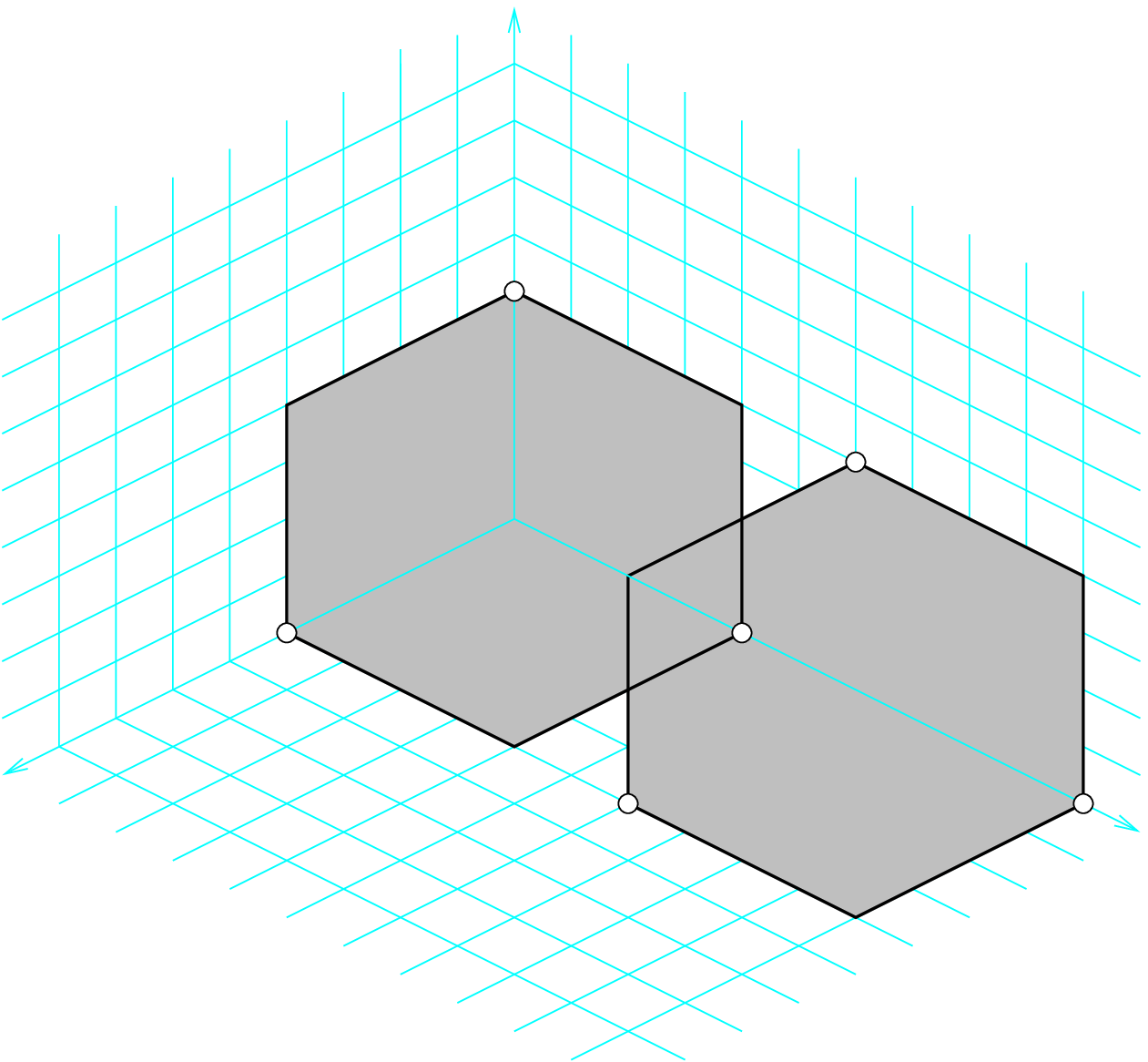} \]
 \caption{In gray, the fundamental cells of $A_1$ (left) and $A_2$ (right).}
 \label{fi:exam1}
 \end{figure}

We can also recover the result using Theorem \ref{Thm:Comb}. We have
\[
A_1 \vee A_2 = \left[ \begin{array}{rrr} 0 & -2 & -2 \\ 1 & 0 & 1 \\ -2 & -2 &
    0 \end{array} \right] \:,
\]
 for which $\rho(A_1 \vee A_2)=0$ and $\Gc(A_1 \vee A_2)=\Gc(A_1)=\Gc(A_2)$, so condition~$(C3)$ of  Theorem \ref{Thm:Comb} holds as well.
\end{example}

\begin{example}
Consider the matrices:
\[
B_1 = \left[ \begin{array}{rrr} 0 & -2 & -2 \\ -2 & 0 & -2 \\ -2 & -2 &
    0 \end{array} \right]\:, \qquad B_2 = \left[ \begin{array}{rrr} 0 & -5 & -1 \\ 1 & 0 & 2 \\ -3 & -6 &
    0 \end{array} \right] \: ,
\]
where $B_1=A_1$ and $B_2=\diag({\bf -v})\odot B_1 \odot \diag({\bf
  v})$ for $\bfv= (0,-3,1)$. Here again, we have
$\urk(B_1)=\urk(B_2)=3$.

\begin{figure}[ht]
\[ \epsfxsize=140pt \epsfbox{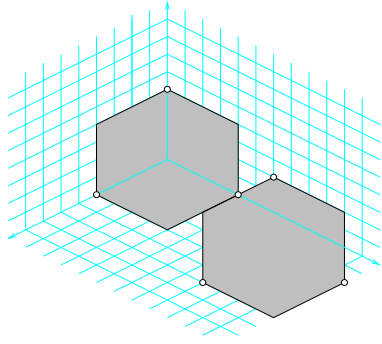} \]
 \caption{In gray, the fundamental cells of $B_1$ (left) and $B_2$ (right).}
 \label{fi:exam2}
 \end{figure}

For each matrix, the fundamental cell is again the interior of the set of eigenvectors.
But now the
intersection of the two fundamental cells is empty, see Figure~
\ref{fi:exam2}. Then $\urk\langle B_1,B_2 \rangle < 3$, by Theorem
\ref{th:geom}, and
one can check that indeed $\trrk(B_1B_2)=2$.

Moreover, we
can prove that $\urk\langle B_1,B_2 \rangle =2$. Indeed, the
intersection of the sets of eigenvectors has topological dimension 2, see Figure
\ref{fi:exam2}.
The common eigenvectors of $B_1$ and $B_2$ are also eigenvectors of
any matrix in $\langle B_1,B_2 \rangle$. Therefore, a matrix in
$\langle B_1,B_2 \rangle$ has a set of eigenvectors of topological
dimension at least 2, hence is not of rank 1.

We can confirm that $\urk\langle B_1,B_2 \rangle <3$ using Theorem~\ref{Thm:Comb}. Consider
\[
B_1 \vee B_2 = \left[ \begin{array}{rrr} 0 & -2 & -1 \\ 1 & 0 & 2 \\ -2 & -2 &
    0 \end{array} \right] \:
\]
for which  $\rho(B_1 \vee B_2)=0$. We have $(2, 3) \in
\Gc(B_1 \vee B_2), (B_2)_{23}= (B_1 \vee B_2)_{23}$, and  $( 2, 3) \not\in
\Gc(B_2)$, so condition~$(C3)$ of Theorem \ref{Thm:Comb} fails.
% \[
% B_1\odot B_2 = \left[ \begin{array}{rrr} 0 & -2 & 0 \\ 1 & 0 & 2 \\ -1 & -2 &
%     0 \end{array} \right] \:,
% \]
% for which the tropical rank is equal to 2 (the identity and the
% transposition of 2 and 3 both realize the maximum in the permanent).
\end{example}

\begin{example}
Consider the matrices:
\[
C_1 = \left[ \begin{array}{rrr} -2 & 0 & -2 \\ -2 & -2 & 0 \\ 0 & -2 &
    -2 \end{array} \right]\:, \qquad C_2 = \left[ \begin{array}{rrr} -1 & 0 & -1 \\ -1 & -1 & 0 \\ 0 & -1 &
    -1 \end{array} \right] \:.
\]
Observe that $C_1\vee C_2 = C_2$. The conditions of Theorem
\ref{Thm:Comb} are clearly satisfied, implying that $\urk\langle
C_1,C_2\rangle =3$.

\begin{figure}[ht]
\[ \epsfxsize=120pt \epsfbox{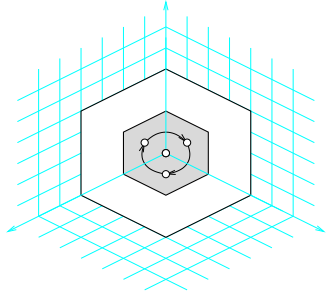} \]
 \caption{The fundamental cells of $C_1$ (white) and $C_2$ (gray).}
 \label{fi:exam3}
 \end{figure}

This can also be checked using Theorem
\ref{th:geom}. The fundamental cells of the two matrices have been
represented in Figure \ref{fi:exam3}. Each matrix acts as a rotation
of angle $-2\pi/3$ and center $(0,0,0)$ on its respective fundamental
cell. In particular, $(0,0,0)$ is the unique eigenvector of both $C_1$ and
$C_2$, and it belongs to the intersection of the fundamental cells.
\end{example}

\begin{example}
Consider the matrices:
\[
D_1 = \left[ \begin{array}{rrr} -2& 0 & -2 \\ -2 & -2 & 0 \\ 0 & -2 &
    -2 \end{array} \right]\:, \qquad D_2 = \left[ \begin{array}{rrr} -1 & 0.2 & -0.8 \\ -1.2 & -1 & 0 \\ -0.2 & -1 &
    -1 \end{array} \right] \:,
\]
where  $D_1=C_1$ and $D_2=\diag({\bf -w})\odot C_2 \odot \diag({\bf
  w})$ for ${\bf w}= (-0.2,0,0)$.
On their fundamental cell, the two matrices act as a rotation
of angle $-2\pi/3$ with respective centers $(0,0,0)$ and
$(0.2,0,0)$. Therefore, they have no common
eigenvector, and hence $\urk \langle D_1,D_2 \rangle <3$ by Theorem \ref{th:geom}.
Furthermore,
\[
D_1 \vee D_2 = \left[ \begin{array}{rrr} -1 & 0.2 & -0.8 \\ -1.2 & -1 & 0 \\ 0 & -1 &
    -1 \end{array} \right] \:,
\]
with $\rho(D_1 \vee D_2) = 0.2/3$, and applying Theorem
\ref{Thm:Comb}, we double-check that we have $\urk \langle D_1,D_2 \rangle <3$.

\begin{figure}[ht]
\[ \epsfxsize=120pt \epsfbox{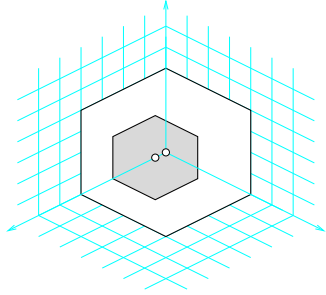} \]
 \caption{The fundamental cells of $D_1$ (white) and $D_2$ (gray).}
 \label{fi:exam4}
 \end{figure}

We check that $\urk(D_1D_2D_1)=1$ so we have $\urk \langle
D_1,D_2 \rangle =1$.
(This last result can also be recovered using \cite[Theorem
7.3.1]{mair95}.)
\end{example}

\subsection{Projectively bounded semigroups}

In obtaining Theorems~\ref{Thm:Comb} and~\ref{th:geom}, the assumption that the semigroup of matrices is finitely generated is used only twice. First, to prove Lemmas~\ref{pr:burnside} and \ref{pr:def}, second, to define matrix~$M$ in Theorem~\ref{Thm:Comb}.
So we obtain the following extension for free. 

\begin{definition}
A subset $\gen$ of non-null matrices in $\Tnn$ is {\em projectively bounded} if 
$\sup_{A\in \gen} \max_{A_{ij},A_{k\ell}\in\R}|A_{ij}-A_{k \ell}|<+\infty$. 
\end{definition}

Theorems~\ref{Thm:Comb} and~\ref{th:geom} hold if we replace in the statements: (i) ``finite set $\gen$" by "projectively bounded set $\gen$''; (ii) ``Set $M=\bigvee_{A\in \gen} \widetilde{A}$''  by 
``Set $M=\sup_{A\in \gen} \widetilde{A}$'' (for Theorem \ref{Thm:Comb}); (iii) ``$\urk(\SG)=n$'' by 
``$\forall P \in \SG, \urk(P)=n$''.

\section{{\bf Summary and open issues}}\label{sec:summery}

Next table summarizes the known results for the time-complexity of the
two basic questions concerning the rank.%   of a matrix or matrix
% semigroup.

% Beside not being coincide, concerning computational time-complexity,
% the diverse notions of rank of tropical matrices (cf. Definition
% \ref{de:ranks}) also have different behavior. The table below
% summarizes the main  notions of rank that are within the interest of
% this paper.
% % and contrasted landscape:

\medskip

\begin{center}
\begin{tabular}{l|l|l}
& Deciding if maximal & Computing \\[2mm] \hline && \\
   $\trrk(A)$ & Polyn.~\cite{BuHe,butk94}  & NP-hard~\cite{KiRo05}
   \\[2mm] \hline &&  \\
   $\clrk$/$\rwrk (A)$ & Polyn.~\cite[\S3.4]{Butkovic} & Polyn.~\cite[\S3.4]{Butkovic}
   \\[2mm] \hline && \\
    $\urk(A)$ & Polyn., Cor. \ref{co:pol1} & Polyn., Cor. \ref{co:pol1}
    \\[2mm] \hline &&  \\
     $\urk(\SG)$ &  Polyn., Cor. \ref{co:pol2} & \qquad ?? \\ [2mm]
\end{tabular}
\end{center}

\medskip

Two points are worth emphasizing. First, computing the ultimate rank of a matrix is of
polynomial time-complexity, while  computing the initial rank of a
matrix could be  NP-hard. Second, the general question about the
decidability of computing
the ultimate rank of a matrix semigroup is still open.

\begin{problem}
How to compute the ultimate rank of a finitely generated matrix semigroup?
\end{problem}

In the present paper, we have proved that the problem of determining
whether the ultimate rank
is maximal is solvable in polynomial time. Another solvable case is that of
semigroups of ultimate rank $0$. Indeed, the case of rank $0$ is
equivalent to mortality, which is known to be decidable and NP-complete~\cite{BlTs97}.

Before approaching the general case of arbitrary ultimate rank, an
intermediate important case is that of ultimate rank $1$.
This case is generic~\cite{merl04}, more precisely, if two generators are
``chosen randomly'', the probability that the semigroup has ultimate
rank 1 is equal to 1.
Also some sufficient conditions for ultimate rank 1 are given in \cite[Chapter 7]{mair95}.
Yet, the decidability of the question ``Is the ultimate rank equal to
$1$?'' is still open.

\bibliographystyle{abbrv}
%\bibliography{../../bib/dfz}

\end{document}